\documentclass{article}

\usepackage{lipsum}
\usepackage{amsfonts}
\usepackage{graphicx,url,color}
\usepackage{epstopdf}
\usepackage{algorithmic}
\usepackage{mathrsfs}
\usepackage[colorlinks=true]{hyperref}
\usepackage{pdfsync}

\ifpdf
  \DeclareGraphicsExtensions{.eps,.pdf,.png,.jpg}
\else
  \DeclareGraphicsExtensions{.eps}
\fi

\usepackage{amsmath,amssymb,amsthm}
\usepackage{lmodern,tikz}
\usepackage{enumitem}

\renewcommand{\geq}{\geqslant}
\renewcommand{\leq}{\leqslant}

\newtheorem{theorem}{Theorem}  
\newtheorem{proposition}{Proposition}

\theoremstyle{definition}
\newtheorem{definition}{Definition}
\newtheorem{lemma}{Lemma}

\theoremstyle{definition}
\newtheorem{remark}{Remark}

\newcommand{\R}{\mathbb{R}}

\newcommand{\al}{\alpha}

\newcommand{\FR}[1]{#1}

\newcommand{\Pt}[1]{\left({#1}\right)}
\newcommand{\eps}{\varepsilon}

\newcommand{\grad}{\nabla}

\newcommand{\albe}{\sqrt\frac{4\alpha^{3}}{27\beta} }

\graphicspath{{Figs/}}

\title{Controlling swarming models towards flocks and mills}

\author{Jos\'e A. Carrillo\thanks{Mathematical Institute, University of Oxford, Oxford OX2 6GG, UK (\texttt{carrillo@maths.ox.ac.uk}).}
	\and
	Dante Kalise\thanks{School of Mathematical Sciences, University of Nottingham, UK (\texttt{dante.kalise@nottingham.ac.uk}).}
	\and
	Francesco Rossi\thanks{Dipartimento di Matematica ``Tullio Levi-Civita'', Universit\`a\ degli Studi di Padova, Via Trieste 63, 35121 Padova, Italy  (\texttt{francesco.rossi@math.unipd.it}).}
	\and
	Emmanuel Tr\'elat\thanks{Sorbonne Universit\'e, CNRS, Universit\'e de Paris, Inria, Laboratoire Jacques-Louis Lions (LJLL), F-75005 Paris, France (\texttt{emmanuel.trelat@sorbonne-universite.fr}).}
}

\usepackage{amsopn}

\ifpdf
\hypersetup{
  pdftitle={Controlling swarms \\ towards flocks and mills},
  pdfauthor={J.A. Carrillo, D. Kalise, F. Rossi and E. Tr\'elat}
}
\fi

\begin{document}
	\maketitle
	
	\begin{abstract}
		Self-organization and control around flocks and mills is studied for second-order swarming systems involving self-propulsion and potential terms. It is shown that through the action of constrained control, it is possible to control any initial configuration to a flock or a mill.  The proof builds on an appropriate combination of several arguments: LaSalle invariance principle and Lyapunov-like decreasing functionals, control linearization techniques and quasi-static deformations. A stability analysis of the second-order system guides the design of feedback laws for the stabilization to flock and mills, which are also assessed computationally.

	\end{abstract}
	
	\section{Introduction}\label{intro}
	We analyse the controllability of the interacting particle system of $N$ agents on the plane, governed by second-order dynamics 
	\begin{equation}\label{syst}
			\begin{split}
				\dot x_i(t) &= v_i(t) \\
				\dot v_i(t) &= (\alpha-\beta \vert v_i(t)\vert^2 ) v_i(t) - \frac{1}{N} \sum_{\substack{j=1\\ j\neq i}}^N \nabla W(x_i(t)-x_j(t)) + u_i(t)\,,
			\end{split}
	\end{equation}
	where $x_i(t)\in\R^2$ (resp., $v_i(t)\in\R^2$) is the position (resp., the velocity) of the $i^\textrm{th}$ agent. 
	In this model, the term $ (\alpha-\beta \vert v_i\vert^2 ) v_i$, where $\alpha\geq0$ and $\beta>0$ are fixed, represents a \emph{self-propulsion} force, while
	\begin{equation}\label{syst2b}
	F_i(x) = \frac{1}{N} \sum_{\substack{j=1\\ j\neq i}}^N \nabla W(x_i-x_j) , \qquad F=(F_1,\ldots,F_N)^\top\,,
	\end{equation}
	expresses an \emph{attraction-repulsion} force through the pairwise interaction potential $W$. The control $u=(u_1,\ldots,u_N)$, with $u_i(t)\in\R^2$, is subject to the constraint $\Vert u(t)\Vert\leq M$ for almost every $t\in\R$, where $M>0$ is fixed. Here, $\vert\cdot\vert$ is the Euclidean norm in $\R^2$ and $\Vert\cdot\Vert$ is the $\infty$-norm in $(\R^2)^{N}$ or $(\R^2)^{2N}$ associated to $|\cdot|$, i.e.,
	$$
	\Vert v\Vert=\max_{i=1,\ldots N} |v_i|,\qquad
	\Vert (x,v)\Vert=\max_{i=1,\ldots N} |x_i|+\max_{i=1,\ldots N} |v_i|.
	$$
	
	{Solutions of the control system \eqref{syst} need to be interpreted in the Caratheodory sense, see, e.g., \cite{BressanPiccoli,Sontag}. Existence and uniqueness of the solution is classical, provided that $u(t)\in L^\infty([0,T];\R^N)$: this condition will always be satisfied, as we deal with bounded controls.}
	
	Models of the form \eqref{syst} are particular examples of agent-based models (ABMs). ABMs appear in biology, mathematics, physics, and engineering in order to describe the motion of a collection of $N$ individual entities at the microscopic scale interacting through simple rules. These types of models have been proposed to describe the flocking of birds \cite{Camazine,lukeman,science1999}, the schooling of fish \cite{Aoki,alethea,Bjorn,Hemelrijk:Hildenbrandt,Huth:Wissel},  and swarms of bacteria \cite{KochWhite}, among others. We refer to the surveys \cite{CFTV,survey} for more general models in the area of interacting particle systems in collective behavior. 
	
	Model \eqref{syst} was introduced in \cite{Levine:2000} and extensively studied in \cite{bertozzi,Dorsogna2,Dorsogna} giving a detailed description of patterns and stability properties of particular solutions through numerical experiments. The role of the self-propulsion term of strength $\alpha>0$ versus friction of strength $\beta >0$ is to fix a typical cruise speed for agents. In fact, in the absence of interactions $W=0$ (no potential) and $u=0$ (no control) in \eqref{syst}, except for the unstable equilibrium $v=0$, all trajectories converge to $\vert v_i\vert^2=\frac{\alpha}{\beta}$ along heteroclinic orbits. These terms will actually promote the appearance of particular solutions : \emph{flock} and  \emph{mill} solutions (defined below). Even more complicated solutions as double-mills have been studied in the literature \cite{Dorsogna3}. 
	
	We assume throughout the article that the interaction potential is radial $W(x)=U(\vert x\vert)$, with $U$ of class $C^2$ except possibly at the origin, and that the interactions are negligible for large distances $\lim_{r\to+\infty} U'(r)=0$. Typical potentials used in previous works are Morse potentials of the form $U(r) =- C_A e^{-r/\ell_A} + C_R e^{-r/\ell_R}$, the index A standing for ``attractive" and the index R for ``repulsive". As shown in \cite{bertozzi, Dorsogna}, the interesting region is when $\ell=\frac{\ell_R}{\ell_A}<1$ and $C=\frac{C_R}{C_A}>1$. In this case, the derivative $U'$ of the potential is such that $\vert U'\vert$ is bounded, $U'$ is positive up to some $r_0>0,$ and is then negative and converging to $0$ as $r\rightarrow +\infty$. We will refer to this kind of potentials as bounded repulsive-attractive potentials. Other potentials of interest are power-law potentials \cite{KSUB,BCLR} given by $U(r)=\frac{\vert r\vert^a}{a}-\frac{\vert r\vert^b}{b}$ with {$a>b>0$}, for which we have $U'(0)=-\infty$, thus avoiding collisions due to an increasing repulsion whenever two particles get closer. We will refer to this kind of potentials as unbounded repulsive-attractive potentials.
	
	Another family of ABMs of interest arises when introducing alignment mechanisms in the modeling. A basic example of this family is the Cucker-Smale model introduced in \cite{CuckerSmale1,CuckerSmale2} and further developed in \cite{CFRT,HaLiu,HaTa,mt}, among others. The main phenomenon in those models is the emergence of alignment, i.e., consensus in velocity. Imposing consensus in velocity has also been analysed from the point of view of control \cite{BFK15,CFPT2013,CFPT2015,HK18}. These consensus models also have applications in swarm robotics \cite{CKPP19,CKP21}, social and pedestrian dynamics \cite{ACFK17,ABCK,KZ20,Suttida2}  where control theory is applied with different regulation objectives expressed in both ad-hoc and optimal control designs \cite{albi2021momentdriven,ALBI201886,BAILO20181,Suttida1}.
	
	Despite the simplicity of the model \eqref{syst}, a striking phenomenon regarding the long-time asymptotics of solutions occurs. There are several stable self-organized patterns that emerge from these dynamics depending on the initial data even for the same parameter values and interaction potentials \cite{Dorsogna2,Dorsogna}. More precisely, flock and mill solutions, which are relevant examples of self-organized configurations for the swarming model, appear asymptotically. Flock and mill solutions are not equilibria in the classical sense (with $\dot x_i=\dot v_i=0$ for all $i=1,\ldots,N$), but are rather solutions with specific invariance properties inherited from \eqref{syst}. Flock solutions describe configurations with agents moving by uniform translation: a \emph{flock} is a trajectory $(x(t),v(t))=( x^*+t v^*, v^*)$ in which {$x^*=(x^*_1,\ldots, x^*_N)$ is a vector of $N$ positions and $v^*=( v^*_1,\ldots, v^*_N)$ is a vector of $N$ identical velocities $v^*_1=\ldots=v^*_N$ that moreover satisfy $|v^*_i|=\sqrt{\frac{\alpha}{\beta}}$}. A \emph{flock ring} is a flock in which the position of agents $x_i$ are equally distributed on a circle with a certain radius $R$, i.e., {$x^*=(x^*_1,\ldots, x^*_N)$ is of the form $x^*_i=c+\hat x_i$, where $c\in \R^2$ is the center of the circle and $\hat x_1,\ldots,\hat x_N\in \R^2$ for $i=1,\ldots, N$ are equispaced points on a circle of radius $R$ centered at zero, i.e.
	\begin{equation} \label{e-hatx}
	\hat x_i= R \, \mathcal{R}_\theta \begin{pmatrix}\cos\left(\frac{2\pi i}{N}\right) \\[2mm] \sin\left(\frac{2\pi}{N i}\right)\end{pmatrix},\quad \mathcal{R}_\theta=\begin{pmatrix} \cos\theta& -\sin\theta\\ \sin\theta & \phantom{-}\cos\theta \end{pmatrix}.
	\end{equation}
	for some angle $\theta$.}
	
	A \emph{mill} for \eqref{syst} corresponds to $N$ agents rotating with a constant angular velocity $\omega$ with respect to a center $c$, i.e., 
	$(x_i(t),v_i(t))=(\mathcal{R}_{\omega t}( x^*_i-c)+c,\mathcal{R}_{\omega t}  v^*)$
	for some $(x^*,v^*)\in\R^{4N}$. A \emph{mill ring} is a mill in which the position of agents $x_i$ are moreover equally distributed on a circle with a certain radius $R$, i.e., { $x^*_i=c+\hat x_i$ with $\hat x_i$ given by \eqref{e-hatx}}.  As a consequence, velocities satisfy 
	$$
	v_i^*=\frac{1}{R}\sqrt{\frac{\alpha}{\beta}} \, \hat x_i^\perp
	\qquad\textrm{with}\qquad
	\hat x_i^\perp=\mathcal{R}_{\pi/2} \hat x_i .
	$$
	
	{The linear stability properties of flock and mill rings have been studied in detail in \cite{bertozzi,KSUB}. The nonlinear stability analysis of general flocks and mills has been fully analysed in \cite{ABCV,CHM}. As a consequence, in the vicinity of certain flocks/mills, it can be shown that the system self-organizes towards a flock/mill. The precise notion of vicinity is referred to the mentioned literature.} Notice that the space positions of flock solutions to \eqref{syst} correspond to stationary states of the first order model
	\begin{equation}\label{eq:firstorder}
		\dot{x}_i=-\sum_{\substack{j=1\\j\neq i}}^N \nabla
		W\left(x_j-x_i\right), \quad i=1,\dots, N.
	\end{equation}
Flock shapes, stationary states for \eqref{eq:firstorder}, for different potentials can have many different shapes even for biologically motivated potentials \cite{KSUB,bertozzi,BUKB,CHM2}, and their regularity heavily depends  on the repulsion strength at the origin, see \cite{BCLR,BCLR2,CDM}. Characterizing all possible mill and flock profiles for a given potential is equivalent to characterizing all possible stationary states of the first order system \eqref{eq:firstorder} or related equations. This difficult problem has not been solved except in very particular choices of the parameters for the power-law potentials. It can be shown that for a repulsive-attractive potential there is a unique flock and mill ring solution and that the radius is characterized uniquely by balances of the relevant forces: attraction, repulsion and centrifugal forces \cite{BCLR,ABCV}. However, showing that they are the unique flock or mill solution is a challenging problem. It is possible to find potentials for which stable mills exist and they are not rings by numerical experiments. Moreover, some compactly supported potentials generically allow the existence of \emph{flock and mill clusters}, i.e., clusters of particles in which each cluster is a mill, or different flocks in separate directions. For the specific case of flock and mill rings, the radius $R$ is characterized by being a solution to
	\begin{equation}\label{e-radius}
		\sum_{p=1}^{N-1}\sin\Pt{\frac{p\pi}N}\tilde U'\Pt{2R\sin\Pt{\frac{p\pi}N}}=0\,,
	\end{equation}
	where $\tilde U(r)=U(r)-\omega^2\frac{r^2}2$, see \cite{bertozzi,ABCV,CHM}. Flock solutions correspond to  $\omega=0$.
	
	Our main goal in this work is to show that constructive controls can be designed to steer the system from any initial data to these various self-organized configurations in interacting particle systems of collective behavior. The main strategy is to prove that the system \eqref{syst} enjoys interesting controllability properties, such as  being able to: steer the system from any initial condition to some/any flock and/or mill; keep the system close to a flock or a mill with an appropriate feedback law; pass from a flock to a mill or conversely. The control $u$ is assumed to satisfy the constraint $\|u\|\leq M$ with the minimal threshold $M>0$ being clearly identified.
	
	The rest of the paper is structured as follows. Section \ref{mainsec} explains in details the main results and the main novelty of our approach: mixing different control techniques with a deep understanding of the heteroclinic connections in these models. In Section \ref{sec_num}, we recall some known stability properties of flocks and mills, together with proposing different feedback designs for transitioning to flock and mill configurations. The control design is guided by controllability results, local stability properties and heteroclinic connections, and is enriched by the use of numerical optimal controls. {Even if the main contribution of the paper consists of the results in Section \ref{mainsec}, which characterize the controlled transition between different flocks and mills, as the proofs indicate, these transitions are constructed in steps as a concatenation of different feedback controls. Section \ref{sec_num} goes beyond the existence of these controls and provides concrete constructions for such feedback laws, making use of both  stability properties of the system and optimal control elements.} In Sections \ref{sec_proof_mainthm} and \ref{sec_proof_thm2}, we provide a complete and detailed proof of Theorems \ref{mainthm} and \ref{mainthm_variant}.
	
	
	\section{Main results}\label{mainsec}
	
	Our two main results, Theorem \ref{mainthm} in Section \ref{sec_bounded_interactions}, and Theorem \ref{mainthm_variant} in Section \ref{sec_unbounded_interactions}, deal with bounded and unbounded interaction potentials (as defined in the introduction), respectively.
	For each statement, we present a brief sketch of our control strategies with a full proof in the last sections.

	\subsection{First main result: bounded interactions}\label{sec_bounded_interactions}
	In our first main result for bounded repulsive-attractive potentials, we assume that our potential is radial, of class $C^2$, with bounded interactions, i.e., $\sup_{r\geq 0} U'(r)<+\infty$ and negligible interactions at $\infty$, i.e., $\lim_{r\to+\infty} U'(r)=0$. Since the potential is $C^2$ at the origin and radial then $U'(0)=0$. This ensures classical well-posedness of \eqref{syst} (see Remark \ref{rem_potential} further).
	
	\begin{theorem}\label{mainthm}[Control for bounded interactions]
		Let $U:[0,+\infty)\rightarrow\R$ such that $W(x)=U(|x|)$ is of class $C^2$ satisfying $U'(0)=0$, $\lim_{r\to+\infty} U'(r)=0$ and:
		\begin{enumerate}[label=$\bf (U_{\arabic*})$]
			\item\label{U1} $C^2$-boundedness: there exist a positive constant $C$ such that $|U(r)|+|U'(r)|+|U''(r)|<C$ for every $r\in[0,+\infty)$.
		\end{enumerate}
		If the upper bound $M$ for the control action is such that
		\begin{equation}\label{hyp_M}
			M>  M_{\alpha,\beta}=\sqrt\frac{4\alpha^{3}}{27\beta} 
		\end{equation}
		then, given any $\bar v\in\R^2$ such that $\vert\bar v\vert = \sqrt{\frac{\alpha}{\beta}}$, {there exists $\bar x=(\bar x_1,\ldots,\bar x_N)\in \R^{2N}$ such that} the control system can be steered, in sufficiently large time and with a feedback control, from any initial condition to any neighborhood of the flock $(\bar x_1+t\bar v,\bar v),(\bar x_2+t\bar v,\bar v),\ldots,(\bar x_N+t\bar v,\bar v)$ with a control satisfying $\|u\|\leq M$.
		
		\noindent Denoting by $M_F=\sup_{r>0} \vert U'(r)\vert$. Under the stronger assumption
		\begin{equation}\label{hyp_M_stronger}
			M >  \max \left( M_{\alpha,\beta} , \ M_F \right) ,
		\end{equation}
		the control system can be steered, in sufficiently large time and with a feedback control, from any initial condition to any neighborhood of any flock, flock ring, mill, mill ring.
	\end{theorem}
	
	\begin{remark}\label{rem_potential}
	The result can be generalized to potentials satisfying $U'(0)<0$, i.e., repulsive at $0$ but nonsmooth at $0$: for instance, in the case of the widely used Morse potentials \cite{Dorsogna} in the biologically reasonable, one has bounded interactions up to $0$ even if the potential is not $C^2$ at the origin. In Theorem \ref{mainthm} above, we have assumed that $U'(0)=0$ to ensure { that $\nabla W(x_i-x_j)$ is Lipschitz with respect to $x$, hence existence and uniqueness of solutions to \eqref{syst} in the Caratheodory sense is guaranteed (see \cite{BressanPiccoli,Sontag})}. In the more general case where $U'(0)<0$, collisions may occur in finite time and thus well-posedness is not a priori ensured. {Anyway, it is always possible to slightly modify our feedback controls in order to avoid the origin, and thus avoid problematic points for local well-posedness, so essentially we can always assume without loss of generality that the radial potential is of class $C^2$ at the origin if it has bounded interactions, i.e., if $U'(0)$ is finite or $M_F<\infty$.} 
	\end{remark}

	\begin{remark}[The role of the assumptions]
		Condition \eqref{hyp_M} means that the control can counteract the natural tendency of the system to stabilize $|v|$ to $\sqrt\frac\alpha\beta$.
		The value $M_{\alpha,\beta}$ corresponds to the maximum of the function $s\to \alpha s-\beta s^3$ on the half line $s\geq 0$, which is attained at $v=\sqrt{\frac{\alpha}{3\beta}}$.  Under the condition \eqref{hyp_M_stronger}, one can moreover counteract the potential interactions, which allows us to design controls steering the system to any mill. Actually, when the potential allows the existence of cluster mills {(i.e., separated groups of agents, each of them forming a mill)}, it is even possible to steer the system to any such cluster mill configuration.
	\end{remark}

	The control strategy to achieve these various objectives can be made explicit and we will provide it in a \emph{feedback} form, making it particularly convenient to implement in practice. It can even be provided in a \emph{componentwise sparse feedback} form, provided that $M$ is large enough. Hereafter, we explain the controllability strategy, by sketching the proof of Theorem \ref{mainthm} (full detail of the proof is given in Section \ref{sec_proof_mainthm}).
	Given any $\eps>0$, we set
	$$
	\Omega_\eps = \{ (x,v)\in\R^{4N}\ \mid\ v=0,\ \|F(x)\|\leq \eps \} .
	$$
	When $u=0$, the set $\Omega_0$ (for $\eps=0$) is invariant under the dynamics. The $x$-components of each point of $\Omega_0$ are indeed the flocks of the uncontrolled system. Note that, for $\eps>0$ small, $\Omega_\eps$ consists both of topological neighborhoods of $\Omega_0$ in the $x$-variable and of some components ``at infinity'', where $F$ is small because $\lim_{r\to+\infty} U'(r)=0$.
	
	\medskip
	
	\paragraph{\bf Step 1: reaching $\Omega_\eps$} The first step of our strategy consists of steering the system from its given (arbitrary) initial condition to the set $\Omega_\varepsilon$, {where $\eps>0$ is a fixed parameter. It needs to be chosen sufficiently small to ensure ``good approximations'' of the nonlinear dynamics \eqref{syst} by the linearized one}. With this goal, we use the Jurdjevic-Quinn method \cite{JQ}, which is a very powerful approach in control design to derive feedback controls which moreover enjoy instantaneous optimality properties (see \cite{CFPT2013,CFPT2015} as well as \cite{CPRT2,CPRT_A2017} for its application to Cucker-Smale multi-agent models). This is done by differentiating with respect to time an appropriate Lyapunov-type functional, then choosing adequately the feedback control, and finally using arguments close to the LaSalle invariance principle.
	
	\medskip

	\subparagraph{\bf Step 1.1: Jurdjevic-Quinn stabilization}
	Defining the Lyapunov functional 
	$$
	V = \frac{1}{2} \sum_{i=1}^N \vert v_i\vert^2 + \frac{1}{2N} \sum_{\substack{i,j=1 \\ i\neq j}}^N W(x_i-x_j) ,
	$$
	along the trajectories we have
	$$
	\dot V = \sum_{i=1}^N (\alpha-\beta \vert v_i\vert^2 ) \vert v_i \vert^2 + \langle v_i, u_i\rangle\,,
	$$
	which leads us to define appropriate feedback controls $u_i$ making $V$ decrease: roughly speaking, we would like to take $u_i=-M\frac{v_i}{\vert v_i\vert}$ when $\vert v_i\vert\leq\sqrt\frac{\alpha}{\beta}$ and $0$ otherwise. Then, using arguments similar to those used in the LaSalle invariance principle, we obtain convergence to $\Omega_\varepsilon$. 
	Although this gives the main idea, the complete argument in Section \ref{sec_proof_step1} is not so easy. The main difficulty is that $V$ is not proper  (i.e., $V$ is not infinite at infinity), even taking the quotient with respect to translations. This does not ensure that the system asymptotically reaches a neighborhood $\Omega_0$ as in the classical LaSalle principle, as it can also converge to a component at infinity.
	
	\medskip

	\subparagraph{\bf Step 1.2: Reaching $\Omega_\eps$ in finite time} 
	Once $(x,v)$ is sufficiently close to $\Omega_\eps$ at time $T$, both velocities and forces are small. We then apply the control 
	$$
	u_i=-(\alpha -\beta |v_i|^2)v_i+F_i(x)-\eta\frac{v_i}{|v_i|}
	$$
	with $\eta>0$ small to steer each component of $v(T)$ such that $v_i\left(T+\tfrac{|v_i(T)|}{\eta}\right)=0$. Note that $\vert u_i\vert$ remains small because we are near $\Omega_\eps$. By applying this control to each agent, we reach $\Omega_\eps$ in finite time.
	
	\medskip

	\paragraph{\bf Step 2: Local controllability near $\Omega_\eps$} 
	At the end of Step 1, the system is in $\Omega_\eps$. Since $v=0$ and $F(x)\simeq 0$ there, we can set 
	\begin{equation}\label{controlnear}
		u_i=-(\alpha-\beta \vert v_i\vert^2 ) v_i + F_i(x)+w_i
	\end{equation}
	and consider $w$ as a new control, subject to the constraint $\Vert w\Vert\leq M/2$, and thus focus on the very simple control system
	\begin{equation}\label{very_simple_control_system}
			\dot x_i = v_i , \qquad \dot v_i = w_i ,
	\end{equation}
	near $v=0$. It is obvious to generate feedback controls $w$, satisfying $\Vert w\Vert \leq M/2$, steering this simplified control system from any point in $\Omega_\eps$ to any other point in a local neighborhood. Note that $\|u\|\leq M$, since the agents remain near $\Omega_\eps$.
	Notice that we cannot assure at this stage that two disconnnected components of $\Omega_\eps$ can be joined by controls of this form, since the size of the control $M$ is disconnected from the size of the force, due to the interaction potential $U$. In order to choose particular spatial configurations such as flocks or mill rings we control the interaction potential in the next steps by increasing the value for the constraint $\|u\|\leq M$ to $M>M_F$.
	
	\medskip

	\paragraph{\bf Step 3: reaching flocks and mills}
	
	\subparagraph{\bf Step 3.1: Reaching flocks} Here, we are still under the assumption $M>  M_{\alpha,\beta}$. It follows from Step 2 that we can steer the control system to a point near $\Omega_\eps$ which is such that $F_i(x)\simeq 0$ and $v_i=\nu\bar v$ for every $i\in\{1,\ldots,N\}$ for some $\bar v$ (that we can choose arbitrarily) such that $\vert\bar v\vert=\sqrt\frac{\alpha}{\beta}$ and some $\nu>0$ small. In other words, we preliminary place the system in a configuration in which all components $v_i$ are small and equal.  Note that, in this preliminary step, we are not free to choose the spatial components $x_i$ where we want such as a predetermined flock profile. This will be done in Step 3.3 below, once we assume the size of our control $M$ may overcome the total maximal interaction force. 
	
After having set $v_i=\nu\bar v$, we keep the control active but small to counteract interaction forces ($u_i=F_i\simeq 0$) and we let the system evolve. We have $v_i(t)=v_j(t)$ for all $i,j$ and thus $x_i(t)-x_j(t)$ remains constant. Each variable $v_i$ evolves according to $\dot v_i = (\alpha-\beta\vert v_i\vert^2)v_i$, starting at the initial value $\nu\bar v$, and hence $v_i(t)\rightarrow\bar v$ as $t\rightarrow+\infty$. This is a motion along an heteroclinic orbit, see Section \ref{sec_proof_step3}.
	
	Note that, along this trajectory, in case of unstability the motion can be stabilized by using a feedback control: one linearizes the system along the nominal trajectory, and then correct errors by feedback.
	
	\medskip

	\subparagraph{\bf Step 3.2: Passing from a flock velocity to another flock velocity} We now have the control system in a flock and we may want to steer the system to a flock with the same relative positions and a different velocity, without necessarily starting the whole procedure from scratch (i.e., achieve Steps 1, 2, 3.1 again).
	
	This can be done by {using the technique of} quasi-static deformation (see \cite{coron-trelat04,coron-trelat06} and see Section \ref{sec_proof_step3} for details) as follows. 
	{Let us consider the initial flock $(\bar x(t)=\bar x(0)+\bar v^0 t,\bar v^0)$ meaning that all velocities are equal to $\bar v^0$ with $\vert\bar v^0\vert=\sqrt{\frac{\alpha}{\beta}}$. Assuming that the system is near this initial flock, we want to steer it to (or near to) the target flock $(\bar x(t)=\bar x(0)+\bar v^1 t,\bar v^1)$ with $\vert\bar v^1\vert=\sqrt{\frac{\alpha}{\beta}}$. The idea is to deform the flock sufficiently slowly in time in order to be able to compensate for the small errors by an appropriate feedback law. Recall that, at any given flock, we have $(\alpha-\beta \vert v_i\vert^2 ) v_i=0$ and $F_i(x)=0$.
Let us consider a continuous path $\tau\in[0,1]\mapsto\bar v(\tau)$ satisfying $|\bar v(\tau)|=\sqrt\frac\alpha\beta$ for every $\tau \in[0,1]$ and such that $\bar v(0)=\bar v^0$, the initial velocity of the flock, and $\bar v(1)=\bar v^1$, the velocity of the target flock. Follow the path slowly-in-time, by taking $\tau=\kappa t$ with $\kappa>0$ small enough. Now, along this path, the simplified control system is not autonomous linear but is anyway a \emph{slowly-varying} linear control system, obviously satisfying the Kalman controllability condition. Then, by pole-shifting, it is possible to design feedback controls, tracking this path and thus steering the control system, in time $1/\varepsilon$, to any point of a neighborhood of $\bar v_1$. More details are given in Section \ref{sec_proof_step3}.
Note that stabilization by classical pole-shifting under the Kalman controllability condition may fail in general for non-autonomous linear control systems but remains valid if the matrices of the system are sufficiently slowly varying in time (see \cite[Chap. 9.6]{khalil}). 
	
	To summarize, by quasi-static deformation, we can bend the motion of the flock by moving slowly the value of $\bar v$ and thus steer the system to another flock. See a numerical example in Figure \ref{quasi} below.

	\medskip
	
	\subparagraph{\bf Step 3.3: Reaching mills} We now make the additional assumption that $M> M_F$.
	At the end of Step 1, the system is in $\Omega_\eps$. Since $v=0$ there, we can again take the control \eqref{controlnear} and consider $w$ as a new control, but now, in contrast to the previous steps, since $F(x)$ will not remain small, $w$ is subject to the constraint $\Vert w\Vert\leq \eta$ for some $\eta>0$ {sufficiently small to ensure the constraint $\|u\|\leq M$}. We can however still focus on the simplified control system \eqref{very_simple_control_system} near $v=0$. It is obvious to generate feedback controls $w$, satisfying $\Vert w\Vert \leq \eta$, steering this simplified control system from any point with $v\simeq 0$ to { a configuration in which all agents are equidistributed on a circle (eventually promoting a mill ring). More precisely, this means that} all $x_i$ are placed along a circle of radius $R_\textrm{mill}$ (where $R_\textrm{mill}$ is a value of a possible mill) and all speeds are given by $v_i = \nu x_i^\perp$ for some $\nu>0$ small. This can be done either by quasi-static deformation as before, or by optimal control. Afterwards, we choose the control 
	\begin{equation*}
		u_i=F_i-\frac{|v_i|^2}{R_\textrm{mill}} \frac{x_i}{|x_i|}
	\end{equation*}
	to ensure that each agent undergoes the correct centripetal force {(i.e., $\dot v_i=-\frac{|v_i|^2}{R_\textrm{mill}} \frac{x_i}{|x_i|}$)} and moves along the circle of radius $R_\textrm{mill}$. We let the system evolve and observe that $|x_i(t)-x_j(t)|$ remains constant for each pair $i,j$, as in Step 3.1. Moreover, $v_i$ evolves according to 
	$$
	\dot v_i = (\alpha-\beta\vert v_i\vert^2)v_i-\frac{|v_i|^2}{R_\textrm{mill}} \frac{x_i}{|x_i|}
	$$ 
	along an heteroclinic orbit, and $v_i(t)-\sqrt\frac{\alpha}{\beta} \, x_i(t)^\perp \rightarrow 0$ as $t\rightarrow+\infty$, i.e., we have convergence to a mill. Such statements can be easily checked in polar coordinates.
	
	Note that we have steered the system to a mill of which we can choose the center. Moreover, we can even steer the system to some arbitrary mill clusters: to do that, it suffices to choose appropriate mill clusters that are sufficiently far one from each other.

	\medskip

	\subparagraph{\bf Step 3.4: Reaching flock rings} To reach a flock ring, that is, a flock where all agents are equidistributed along a circle of radius $R_\textrm{mill}$, we first proceed as in Step 3.3, except that the target velocity $v_i = \nu x_i^\perp$ is replaced by $v_i = \nu \bar v$ for some $\bar v$ such that $\vert\bar v\vert=1$ (the desired direction of the flock). We then follow Step 3.1.

	\medskip

	\subparagraph{\bf Step 3.5: Passing from any flock or mill to any other} By the same procedure as in Step 3.2, it is now clear that we can pass from any flock (or flock ring) or mill to any other, without having to restart the whole procedure at Step 1.

	\medskip

	\paragraph{\bf Step 4: Sparsification}
	The feedback controls defined above are not componentwise sparse, in the sense that, at any instant of time, several (usually all) components of the control are active. In order to keep a minimal amount of intervention at any instant of time, the notion of componentwise sparse control has been introduced in \cite{CFPT2013, CFPT2015} (see also \cite{PRT2015} for the corresponding notion in infinite dimension), meaning that, at any fixed time, at most one component of the control can be active. This notion models the action of one leader on a group of agents, like a single dog acting on a flock of sheep.
	It has been shown in \cite{CPRT_A2017} how to design, for dissipative control systems, a componentwise sparse feedback control, starting from any feedback control. This can be done, for instance, by applying an averaging procedure. This is however at the unavoidable price of requiring that $M>N M_{\alpha,\beta}$. 
	Anyway, this ``sparsification" procedure is general enough to produce sparse feedback controls, steering the control system \eqref{syst} from any initial condition to any flock or mill, in sufficiently large time.

	\medskip

We conclude with several remarks on extensions of these results.

\begin{remark}[Regularity of the control] All along the proof, we use controls that are either in a feedback form $u_i=u_i(x_i,v_i)$ and or in an open-loop form $u_i=u_i(t)$. In particular, feedback controls appear for stabilization in \eqref{e-controlO} and are always Lipschitz with respect to $(x_i,v_i)$, ensuring existence and uniqueness of Caratheodory solutions of \eqref{syst} (see \cite{BressanPiccoli,Sontag}). Both open-loop and feedback controls can be smoothened and thus can be chosen as $C^\infty$ functions, with no major difference with respect to the results provided here. Indeed, they can be smoothened even with an arbitrarily small increase of the $C^1$ norm, ensuring that the constraint \eqref{hyp_M_stronger} is always satisfied. \FR{This can be seen as a consequence of the fact that, for systems with Lipschitz vector fields,  $L^1$ convergence of controls implies convergence of trajectories, see, e.g., \cite{BressanPiccoli,trelatbook2005}. Indeed a Lipschitz control $u(t,x)$ can always be approximated in $L^1$ by a sequence of $C^\infty$ functions $u^n(t,x)$ satisfying $\|u^n-u\|_{L^1}\to 0$ and the associated solution $(x^n,v^n)$ converges to the solution $(x,v)$.}
\end{remark}

	\begin{remark}[On exact controllability]
	In Theorem \ref{mainthm}, we have established asymptotic feedback controllability to flocks or mills. Since the linearized system around any flock or mill satisfies the Kalman condition (we are here in finite dimension), it follows that the system \eqref{syst} is locally controllable around flocks and mills. Therefore, as soon as the agents are close enough to a flock (or a mill), one can always design an open-loop control steering the system in finite time exactly to the flock (or to the mill). In other words, in the framework of Theorem \ref{mainthm}:
	\begin{itemize}
		\item Under \eqref{hyp_M}, the system \eqref{syst} can be steered in large time to any flock.
		\item Under \eqref{hyp_M_stronger}, the system \eqref{syst} can be steered in large time to flocks or mills.
	\end{itemize}
	This is a global controllability result to flocks and mills.
	\end{remark}
	
	\begin{remark}[On sharpness of the assumptions and black hole phenomenon]
	In Theorem \ref{mainthm}, we have assumed that the value of $M$ is large enough. If $M$ is too small, then it may happen that we do not have a sufficiently strong control to achieve our objectives. Given that, for certain potentials, some of the flocks or mills seem to be strongly attractive, when $M$ is too small we may even fall in the ``black hole phenomenon" (see \cite{PPT2019}). This simply means that the attraction power of a given mill would be too strong to be countered by the control: one then cannot escape from such a basin of attraction. This consideration shows that our assumption on $M$ is, in some sense, unavoidable. However, it is likely that, for some specific classes of potentials, the assumption can be weakened.
	\end{remark}
	
	\begin{remark}[Mean-field limit]\label{rem:mf}
	All the results in Theorem \ref{mainthm} related to feeback controls are still valid for smooth solutions of the mean-field partial differential equation of Vlasov-type obtained as the formal limit $N\to+\infty$ of \eqref{syst} (see \cite{CPRT2,PRT2015}). This is due to the fact that feedback controls are Lipschitz functions of the state (see \cite{PR18}). The singular character of the flock and mill solutions as solutions of the partial differential equation is not a difficulty, since smooth solutions exist globally and only concentrate in the velocity variable as $t\to+\infty$. We do not provide any details.
	\end{remark}

	\subsection{Second main result: unbounded interactions}\label{sec_unbounded_interactions}
	In Theorem \ref{mainthm}, we have assumed that $\vert U'\vert$ is bounded. However, such an assumption does not involve the case of a potential that explodes near $r=0$, i.e., satisfying $U(0)=-\infty$. Such potentials are often considered in collective behavior models, since they reflect the fact that two agents cannot meet: such an assumption rules out shocks.
	In this case, we can refine Theorem \ref{mainthm} as follows.
	
	\begin{theorem}\label{mainthm_variant}
		Let $U:(0,+\infty)\rightarrow\R$ generating a radial potential $W(x)=U(|x|)$ of class $C^2$ except at the origin, $\lim_{r\to+\infty} U'(r)=0$ satisfying:
		\begin{enumerate}[label=$\bf (U_{\arabic*})$]
			\setcounter{enumi}{1}
			\item\label{U2} Repulsiveness at $0$: $\lim_{r\to 0}U'(r)=-\infty$ and for every $R>0$ there exist a positive constant $C(R)$ such that $|U(r)|+|U'(r)|+|U''(r)|<C(R)$ for every $r\in[R,+\infty)$.
		\end{enumerate}
		\begin{enumerate}[leftmargin=*]
			\item If $M>  M_{\alpha,\beta}$,
			then the control system \eqref{syst} can be steered, in sufficiently large time and with a feedback control, to any neighborhood of any flock.
			\item Set $\tilde M_F=\sup_{r>0} U'(r)$ and $\tilde M_N=\sup\{ |U'(r)|\,: \, r>2\sin(\tfrac{\pi}N) \bar R\}$. Under the stronger assumption $M>  \max\left(M_{\alpha,\beta}  + \tilde M_F, \tilde M_N\right)$
			for some $\bar R>0$, the control system \eqref{syst} can be steered, in sufficiently large time and with a feedback control, from any initial condition to any neighborhood of any flock ring with radius larger than $\bar R$.
			\item Under the stronger assumption $
				M>  \max\left(M_{\alpha,\beta}  + \tilde M_F,\tilde M_N+\frac{\alpha}{\beta \bar R}\right) 
			$,
			for some $\bar R>0$, the control system \eqref{syst} can be steered to any neighborhood of any mill ring with radius larger than $\bar R$.
		\end{enumerate}
	\end{theorem}
	
	The proof of the first statement is nearly identical to Steps 1, 2 and 3.1 of the proof of Theorem \ref{mainthm}. Few more estimates are needed to ensure convergence of velocities and forces to $0$ (see details in Section \ref{s-21}).
	
	To prove the two last statements, the differences are more significant. Hereafter, we give a brief sketch of the strategy. Full details are given in Section \ref{s-22}.
	In contrast to Theorem \ref{mainthm}, and since the potential is now infinite at $0$, to prove Theorem \ref{mainthm_variant} the main idea is now to a priori blow-up the group of agents, i.e., by steering them far from each other. Afterwards, we place all of them along an adequate configuration to let them converge to a flock or a mill. The first step is done by creating a fictitious potential, killing the attractive part of the initial potential $U$ and adding a small repulsive part.
	
	\medskip
	
	\paragraph{\bf Step 1: Blow-up} 
	In contrast to Theorem \ref{mainthm}, the set $\Omega_0$ may now be empty (for instance, take the potential $U(r)=1/r$), but in the strategy that we develop below, this is actually an advantage, because in this case the Jurdjevic-Quinn strategy steers at least one of the $N$ particles sufficiently far from all others.

	\medskip
	
	\subparagraph{\bf Step 1.1: Fictitious purely radial potential} 
	We apply the control
	\begin{equation}\label{control_fictitious}
		u_i=\frac{1}{N}\sum_{\substack{j=1\\ j\neq i}}^N \left( \grad W (x_i-x_j)-\grad\tilde W (x_i-x_j) \right) +w_i = F_i(x) - \tilde F_i(x) + w_i
	\end{equation}
	which amounts to replacing in \eqref{syst} the radial potential $U$ with a new potential $\tilde U$, that we choose purely repulsive. This means that we replace the potential $W(x)=U(|x|)$ with $\tilde W(x)=\tilde U(|x|)$. Hence, by choosing this control, we cancel the attractive part of the potential $U$ and we add a small repulsive part. The constraint $\|u\|\leq M$ is satisfied if $\|w\|\leq \tilde M$ with 
	\begin{equation}\label{tildeM}
		\tilde M=M- \sup_{r>0} |\tilde U'(r)-U'(r)|.
	\end{equation}
	In such a way, the control system \eqref{syst} becomes
	\begin{equation}\label{syst2}
		\begin{split}
			\dot x_i(t) &= v_i(t) \\
			\dot v_i(t) &= (\alpha-\beta \vert v_i(t)\vert^2 ) v_i(t) - \tilde F_i(x(t)) + w_i(t) 
		\end{split}
	\end{equation}
	where { $\tilde F$ is the force associated to the potential $\tilde W$} and $w$ is the new control, subject to the constraint $\|w(t)\|\leq \tilde M$ for almost every $t$.

	\medskip
	
	\subparagraph{\bf Step 1.2: Blowing-up all agents}
	We apply Step 1 of Theorem \ref{mainthm} (Jurdjevic-Quinn stabilization procedure) to the modified control system \eqref{syst2}, in order to steer it to the set
	$$
	\tilde \Omega_\eps = \{ (x,v)\in\R^{4N}\ \mid\ v=0,\ \|\tilde F(x)\|<\eps \} .
	$$
	By choosing $\eps>0$ sufficiently small, we thus obtain that one of the agents is moved arbitrarily far from all others, since all forces are repulsive. Moreover, by a simple geometric observation, one can prove that this agent is also far from the convex hull of all other agents, i.e., it can be genuinely moved away from all other $N-1$ agents.
	
	Repeating then the same strategy to all agents, one by one, we ultimately obtain that all particles can be moved far away one from each other.

	\medskip

	\paragraph{\bf Step 2: Circular equidistributed configuration} Since all agents are arbitrarily far, we have $F(x)\simeq 0$ by the assumption $\lim_{r\to+\infty} U'(r)=0$. We can then steer all agents to a circular equidistributed configuration of large radius. The only technical detail is that we need to ensure that all particles keep being sufficiently far one from each other, so that $F(x)\simeq 0$. Details are given in Section \ref{s-22}.

	\medskip
	
	\paragraph{\bf Step 3: Converging to the desired target} We finally steer the configuration to the desired one as follows: we first give the desired initial impulse to the velocity variable to have either a flock ring or a mill ring with a large radius (as in the proof of Theorem \ref{mainthm}, Step 3.1 for flock rings and Step 3.3 for mill rings).
	We then reduce the radius to the desired $r\geq \bar R$ by applying a suitable central force. If the target is a flock ring, we can counteract the interaction forces thanks to the assumption $M>\sup\{ |U'(r)|\ \mid\ r>2\sin(\tfrac{\pi}N) \bar R\}$. If the target is a mill ring, we also need to counteract the centripetal force, which is possible thanks to the assumption $M>\sup\{ |U'(r)|+\frac{\alpha}{\beta \bar R} \ \mid\  r>2\sin(\tfrac{\pi}N) \bar R \}$.

	
	\section{Local stability and feedback control design}\label{sec_num}
	
	In this section, we recall some known stability properties of flocks and mills established in \cite{ABCV,bertozzi,CHM}. These properties are useful for control design since under conditions of local stability of the corresponding flock or mill solutions, our controls can be be switched off once in a neighborhood of the profiles and the free dynamics of the system will self-regulate towards the desired limiting state. We show how to use the strategies depicted in the previous section to build feedback controls that stabilize to the flock or mill solutions.
	
	\subsection{Local stability of flock manifolds} {We now consider the system \eqref{syst} with no control and we study its stability properties.}  First remark that \eqref{syst} with $u\equiv 0$ is obviously invariant under translations and rotations in space, and thus, if we find a flock solution, this particular solution gives rise to infinitely many flock solutions via these invariances. Moreover, once we have a flock solution the direction of the translational velocity of the flock can be freely chosen. It is then natural to deal with flock solutions seen as a manifold of configurations, that we describe in the following 
	
	\begin{definition}\label{d-flockmill} Let $(x^*,v^*)$ be an initial configuration such that the corresponding solution of \eqref{syst} is a flock. The \emph{flock manifold} associated to $x^*$ is
		$$\mathcal{F}(x^*)=\left\{ 
		\begin{pmatrix} x\\ v \end{pmatrix}\in\R^{4N}
		\ \mid\ x\in RT(x^*),\, v=\mathbf{1}_N\otimes \bar v,\, \bar v\in \R^2,\, |\bar v|=\sqrt{\frac{\alpha}{\beta}}
		\right\}
		$$
		where 
		$
		RT(x^*)=\left\{ \mathbf{1}_N\otimes b+\Pt{\mathrm{Id}_N\otimes  \mathcal{R}_\theta} x^* 
		\ \mid\ \theta\in[0,2\pi),\ b\in \R^2\right\}
		$ is the family of states obtained by rotations and translations from $x^*$.
		
		{Here, the symbol $M\otimes K$ denotes the Kronecker product: given a matrix of dimension $n\times m$ and a matrix of dimension $p\times r$, the matrix $M\otimes K$ has dimension $np\times mr$ and is built by replacing each element $m_{ij}$ in $M$ with the matrix $m_{ij} K$.}
	\end{definition}
	
	The invariance of trajectories plays a crucial role in the study of solutions of \eqref{syst}. In particular, the concept of stability needs to be adapted to the fact that solutions do not converge to a precise point: instead, they converge to the manifold in the sense of the distance, while they go to infinity in the $x$-variable. For this reason, we adapt the classical definition of asymptotic convergence to this setting, by considering neighborhoods invariantly defined based on the metric. We denote by
	$
		d((x,v),A)=\underset{a\in A}{\inf}\|(x,v)-a\|
	$ 
the distance to the set $A\subset \R^{4N}$. 
	
	\begin{definition}\label{d-invnei}
		A manifold $A$ is locally asymptotically stable if there exists a $\eps_0$-neighborhood $V_0$ of $A$ such that for each $\eps$-neighborhood $V\subset V_0$ of $A$, there exists a $\delta$-neighborhood $U$ of $A$ such that $(x(0),v(0))\in U$ implies both $(x(t),v(t))=(x_1(t),\ldots,x_N(t),v_1(t),\ldots,v_N(t))\in V$ and $\underset{t\to \infty}{\lim}d((x(t),v(t)),A)=0.$
	\end{definition}

The asymptotic stability of the flock manifold was analysed in \cite{CHM}. There, it was shown that it is intimately related to the stability of the linearized system of the associated first-order system \eqref{eq:firstorder}: $\dot x_i=F_i(x)$ with $F_i$ given by \eqref{syst2b}, around a stationary configuration $\hat x$ given by $\dot h = G(\hat x) h$ with
\begin{equation*}
	G_{ij}(\hat x)=\begin{cases}
		-\sum_{k\neq i} \mathrm{Hess}\, W(\hat x_i- \hat x_k)&\mbox{~~ for ~} i=j,\\
		\mathrm{Hess}\, W(\hat x_i- \hat x_j) &\mbox{~~ for ~} i\neq j.
	\end{cases}
\end{equation*}
Under certain assumptions on the eigenvalues and eigenspaces of $G(\hat x)$, we have a linear system $\dot h = G(\hat x) h$ with a zero eigenvalue of multiplicity $4$, and all other eigenvalues have a negative real part. We refer to \cite[Theorem 1]{CHM} for the exact assumptions needed for local asymptotic stability around the flocking manifold.

\begin{proposition}  Let $\hat x$ be a stationary state for \eqref{eq:firstorder}. Then, the manifold $\mathcal{F}(\hat x)$ is locally asymptotically stable for all configurations $x^*$ under certain assumptions on the the eigenvalues and eigenspaces of $G(\hat x)$, in the following sense: any small enough perturbation in the variables $(x,v)$ of the flock solution $z^*\in \mathcal{F}(\hat x)$ associated to $x^*$ under the dynamics \eqref{syst}  exponentially converges to $\mathcal{F}(\hat x)$.
\end{proposition}

This attractivity property of $\mathcal{F}(\hat x)$ is natural, in view of the fact that local perturbations in $(x,v)$ might introduce rotations and translations, that is the flock solutions are stable under small perturbations leading to another flock with a small deviation in their direction. In the following, we illustrate the use of these stability properties in conjunction with Theorem \ref{mainthm} to provide an effective feedback design for stabilization towards flock solutions.

\begin{figure}[!ht]
	\centering
	\includegraphics[width=0.33\textwidth]{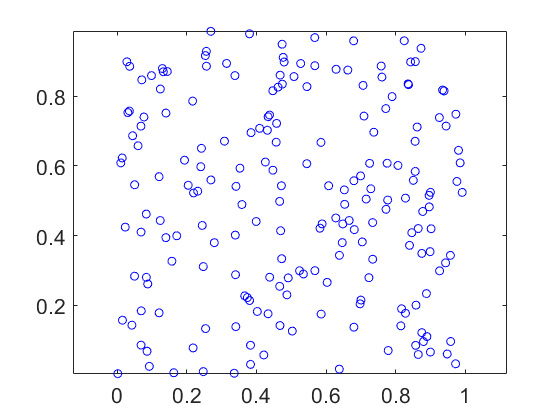}
	\includegraphics[width=0.33\textwidth]{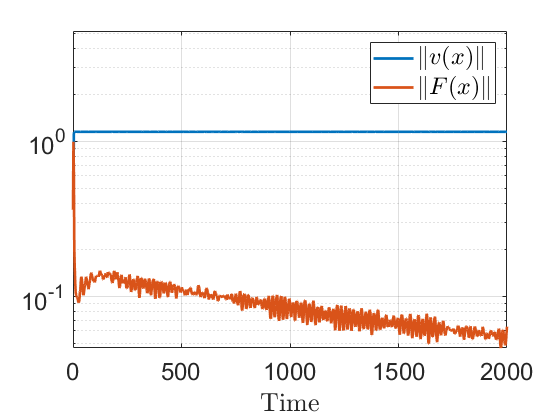}\hfill
	\includegraphics[width=0.33\textwidth]{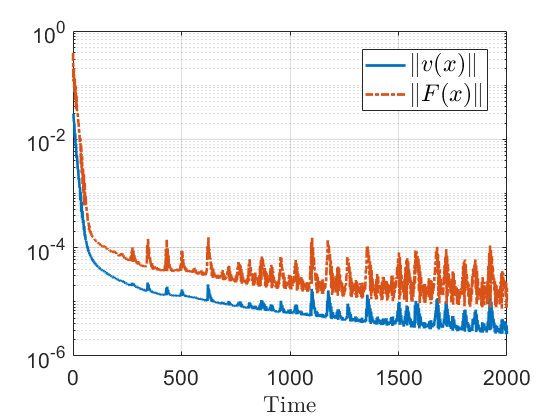}
	\\
	\includegraphics[width=0.33\textwidth]{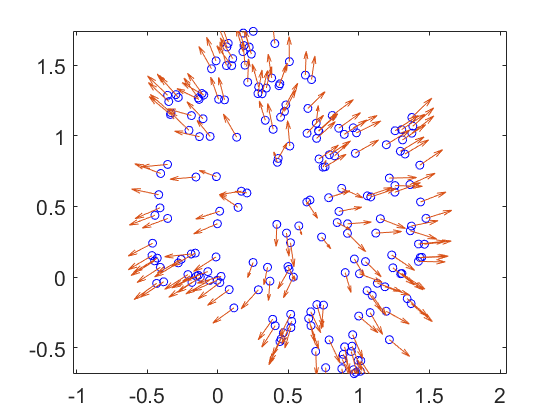}\hfill
	\includegraphics[width=0.33\textwidth]{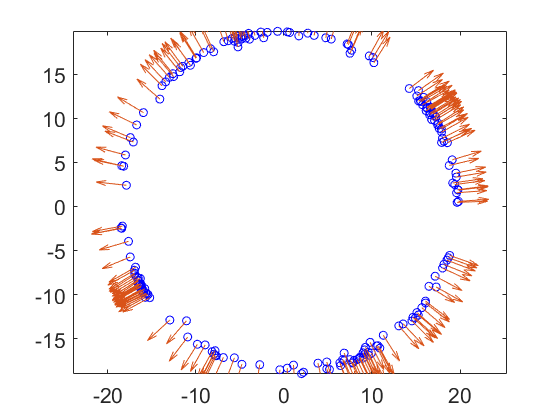}\hfill
	\includegraphics[width=0.33\textwidth]{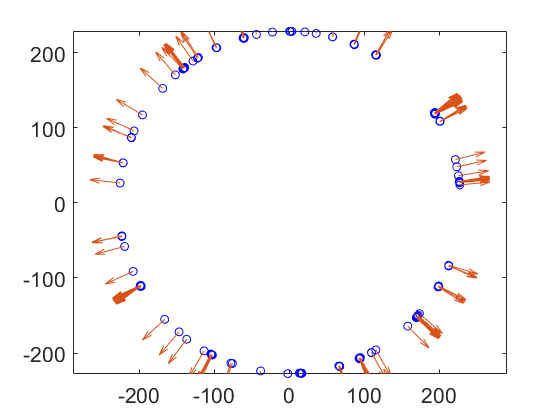}\\
	\includegraphics[width=0.33\textwidth]{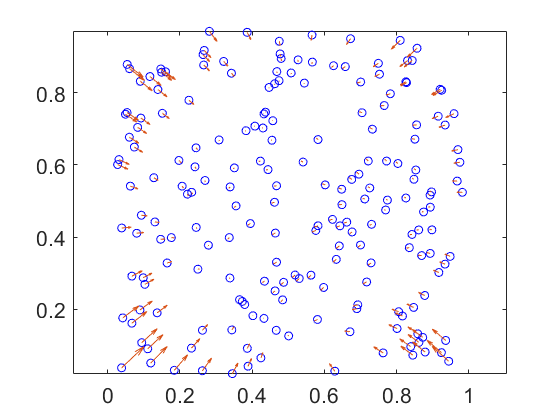}\hfill
	\includegraphics[width=0.33\textwidth]{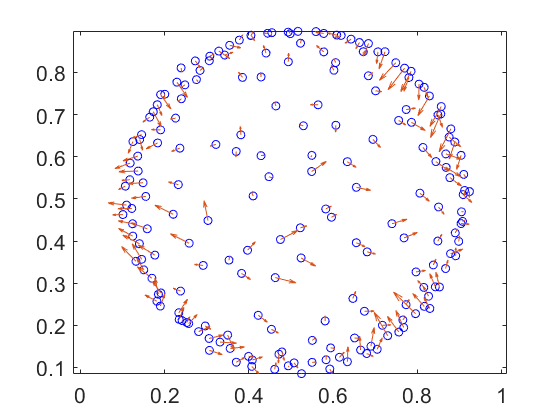}\hfill
	\includegraphics[width=0.33\textwidth]{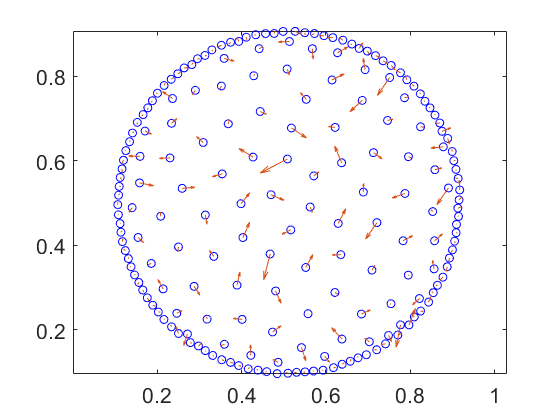}
	\caption{Jurdjevic-Quinn stabilization for the quasi-Morse potential \eqref{qmpot} with $C=0.6, p=1.5$, $l=0.5,\alpha=2,\beta=1.5, N=200$. From left to right. Top: initial state, uncontrolled and controlled energy evolution. Middle: uncontrolled evolution at $t=4,20,200$. Bottom: controlled evolution at $t=4,20,200$. The Jurdjevic-Quinn feedback law stabilizes towards a neighborhood of $\Omega_{\epsilon}$.}\label{step11}
\end{figure}

\subsection{Flock control with a quasi-Morse potential}
Given the potential 
\begin{align}\label{qmpot}
W(x)=U(|x|)=V(|x|)-CV(|x|/l)\,,\qquad V(r)=-e^{-r^p/p}\,,
\end{align}
it is known from \cite{CHM2} that repulsive-attractive forces lead to a locally stable flock manifold with a particular spatial configuration. Figure \ref{step11} provides an illustration of free and controlled dynamics for this system. Given a random configuration of particles at rest (top left), the free dynamics evolve towards a ring formation which grows in time, weakening the influence of the potential (middle row). By implementing the Jurdjevic-Quinn feedback control from Theorem \ref{mainthm}, Step 1.1, the evolution is controlled towards a bounded configuration at rest (bottom left): it corresponds to the locally stable uncontrolled profile found in \cite{CHM2}. The uncontrolled dynamics then converge towards a state where $\|F(x)\|=0$ due to expansion, while $\|v(x)\|$ remains constant because of self-propulsion (top, middle). Instead, the Jurdjevic-Quinn stabilization ensures decay on both quantities (top, right), generating a configuration that can be subsequently stabilized towards different flocks (bottom middle, right).

In Figure \ref{quasi}, we illustrate the quasi-static deformation between different flocks pointing towards different directions using the strategy of Step 3.2 of the previous section. Given a time horizon $T$, initial and terminal velocities $v_0$ and $v_T$, characterized by angles $\theta_0$ and $\theta_T$ respectively, and magnitude $\sqrt{\frac{\alpha}{\beta}}$, we transition from $v_0$ to $v_T$ through the action of the linear, time-dependent, feedback control
\begin{align}\label{quasiflock}
	u_i(v_i,t)=-M(v_i-\mathcal{R}_{\theta(t)}v_0)\,,\quad \theta(t)=\theta_0+\frac{t}{T}(\theta_T-\theta_0)\,,
\end{align}
where $\mathcal{R}_{\theta}$ is the rotation matrix.

	\begin{figure}[!ht]
		\centering
		\includegraphics[width=0.33\textwidth]{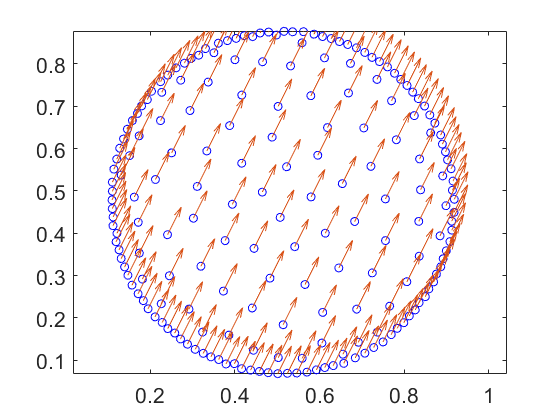}\hfill
		\includegraphics[width=0.33\textwidth]{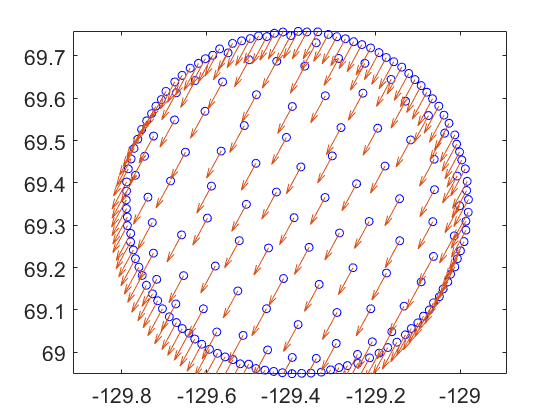}
		\includegraphics[width=0.33\textwidth]{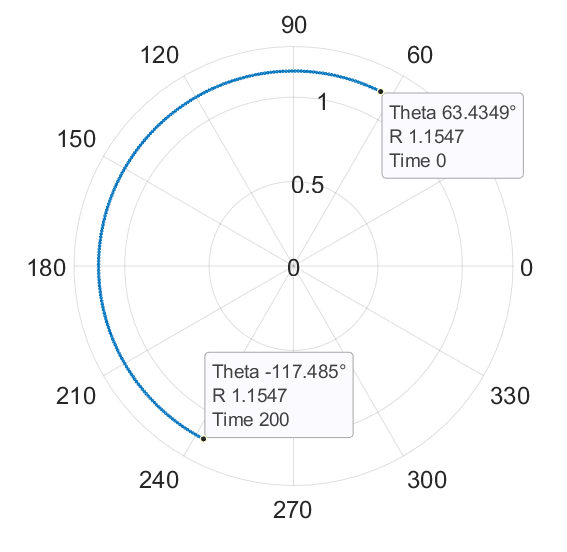}
		\caption{Quasi-static deformation for flock transition with a quasi-Morse potential  \eqref{qmpot} with $C=0.6$, $p=1.5, l=0.5,\alpha=2,\beta=1.5, N=200$. Left: initial flock. Middle: terminal flock. Right: average angle evolution. The use of the quasi-static feedback law \eqref{quasiflock} connects different flocks.}\label{quasi}

	\end{figure}
	
	\subsection{Local stability of mill rings} \label{s-millstab} We now turn our attention to the study of the local behavior of mill rings. We introduce adapted coordinates, and study the asymptotic stability for solutions exhibiting rotational symmetries only. It is then useful to introduce the time-varying orthonormal frame in $\R^{2N}$
$$
	e_i^r=\left(\cos\left(\frac{2\pi i}N+\omega t\right),\sin\left(\frac{2\pi i}N+\omega t\right)\right),\qquad e_i^\perp=(e_i^r)^\perp,
$$
with $e^\perp$ denoting the rotated vector by $\pi/2$ of $e\in\R^2$, and $\omega$ the angular velocity of the mill. Without loss of generality, we consider a mill ring of $N$ agents rotating around the point $(0,0)$. By rearranging indices, we assume that positions and velocities satisfy
$$
	x_i(t)=R e^r_i,\qquad v_i(t)=\sqrt\frac{\alpha}{\beta} e^\perp_i.
$$
for $i=1,\ldots, N$. It is clear that $\dot e^r_i=\omega e^\perp_i$ and $\dot e^\perp_i=-\omega e^r_i$. The mill radius $R$ is given by the solution of \eqref{e-radius}. We now consider a perturbation around such a mill solution, that we write as
	$$
	x_i(t)=(R+r_i(t))\mathcal{R}_{\theta_i(t)}e^r_i,\qquad v_i(t)=\left(\sqrt\frac{\alpha}{\beta}+w_i(t)\right) \mathcal{R}_{\tau_i(t)}e^\perp_i,
	$$
{	by writing:
	\begin{itemize}
	\item $x_i$ in polar coordinates, where $r_i,\theta_i$ are perturbations of the radius and angle variables with respect to $R, \mathrm{arg}(e^r_i)$;
	\item $v_i$ in polar coordinates, where $w_i,\tau_i$ are perturbations of the radius and angle variables with respect to $\sqrt\frac{\alpha}{\beta}, \mathrm{arg}(e^\perp_i)$.
	\end{itemize}
	Here, $\mathrm{arg}(v)$ is the argument of the nonzero vector $v$, i.e., the angle in its polar coordinates.}
	
	A straightforward computation shows that the system \eqref{syst} can be written in term of the new variables $r_i,\theta_i,w_i,\tau_i$. 
A general result of stability around such solutions seems out of reach (see some results for a first-order system with a similar structure in \cite{bertozzi}). We instead briefly investigate the local stability of solutions with rotational symmetries, i.e., solutions that satisfy
	$
	r_i=r_j, \theta_i=\theta_j, w_i=w_j,\tau_i=\tau_j$ with $i,j=1,\ldots, N.
	$
	Notice that due to the rotational symmetry, the last equation becomes $\dot \tau_i=-\omega$.  By dropping the index for variables $r_i,w_i$ and introducing the variable $\gamma=\theta_i-\tau_i$, the system is reduced to
	\begin{equation}\label{e-mill-lin}
		\begin{split}
			\dot r&=\left(\sqrt\frac{\alpha}{\beta}+w \right)\sin(\gamma)\,,\qquad
			\dot \gamma=\left(\frac{\sqrt\frac{\alpha}{\beta}+w}{R+r}-\frac{\phi(r)}{\sqrt\frac{\alpha}{\beta}+w}\right)\cos(\gamma)\,,\\
			\dot w&=-\left(2\sqrt{\alpha\beta}w+\beta w^2\right)\left(\sqrt\frac{\alpha}{\beta}+w \right)-\phi(r)\sin(\gamma)\,,
		\end{split}
	\end{equation}
	where 
	\begin{eqnarray*}
		\phi(r)&=&(1,0)\cdot \frac1N\sum_{j=1}^{N-1}\nabla W\left((R+r)((1,0)-\left(\cos\left(\frac{2\pi j}N\right),\sin\left(\frac{2\pi j}N\right)\right)\right)\\
		&=&\frac1N\sum_{j=1}^{N-1}\sin\left(\frac{\pi j}N\right)U'\left((R+r)\sin\left(\frac{\pi j}N\right)\right)\,.
	\end{eqnarray*}
	Note that, since \eqref{e-radius} is satisfied, we have
	$
	\frac{\sqrt\frac{\alpha}{\beta}}{R}=\frac{\phi(0)}{\sqrt\frac{\alpha}{\beta}}=\omega.
	$
	In particular, the trajectory $(r(t),\gamma(t),w(t))=(0,0,0)$ is a solution of \eqref{e-mill-lin}. This shows invariance of the mill solution under a same translation in all variables $\alpha_i,\beta_i$. One can easily study linear stability properties for \eqref{e-mill-lin}: the linearized system is given by
	\begin{equation*} 
		\begin{pmatrix}
			\dot r\\\dot \gamma\\\dot w
		\end{pmatrix}=A \begin{pmatrix}
			r\\ \gamma\\ w
		\end{pmatrix} .
\mbox{~~~with~~~}A=
		\begin{pmatrix}
			0 & \sqrt\frac{\alpha}{\beta} & 0\\
			-\frac{\omega}R -\frac{\phi'(0)}{\sqrt\frac{\beta}{\alpha}} & 0& \frac{2}{R}\\
			0 & -\phi(0) & -2\alpha
		\end{pmatrix}
			\end{equation*}
{One can compute the characteristic polynomial $\lambda^3+a_2\lambda^2+a_1\lambda^1+a_0$ of $A$, by identifying coefficients with principal minors of $A$. We have
$$
a_2=-\mathrm{Tr}(A)=2\alpha,\qquad 
a_1=\left(\frac{\omega}R +\frac{\phi'(0)}{\sqrt\frac{\beta}{\alpha}}\right)\sqrt\frac{\alpha}{\beta}+\phi(0)\frac2R ,
$$
$$
a_0=-\mathrm{det}(A)=\left(\frac{\omega}R +\frac{\phi'(0)}{\sqrt\frac{\beta}{\alpha}}\right)\sqrt\frac{\alpha}{\beta}(2\alpha) .
$$
Recall a special case of the Routh–Hurwitz criterion for third-order monic polynomials (see \cite{trelatbook2005}): the linear stability property is ensured if and only if $a_2,a_0,a_2a_1-a_0>0$. Then, recalling that $\alpha,\beta,\phi(0)>0$, the linear stability property is ensured by the condition $\frac{\omega}R +\sqrt\frac{\alpha}{\beta}\phi'(0)>0.$}
	This condition is always satisfied for power-law potentials of the form
	\begin{align}\label{plaw}
	U(s)=\frac{|s|^a}a - \frac{|s|^b}{b},\qquad a>b>0 ,
	\end{align}
	 since we have in this case
	$$
	\phi'(0)=\sqrt\frac{\alpha}{\beta}\omega(a-b)R^b\sum_{j=1}^{N-1}\sin\left(\frac{\pi j}N\right).
	$$
	This linear stability analysis shows that the equilibrium solution $(r(t),\gamma(t),w(t))=(0,0,0)$, i.e., the mill ring solution is a locally asymptotically stable equilibrium point to \eqref{e-mill-lin}. As a consequence, the mill ring solution is locally asymptotically stable for perturbations keeping the rotational symmetry of \eqref{syst}. In the following, we explore the interplay between stability of mill rings and control design for power law potentials.
	\begin{figure}[!ht]
		\centering
		\includegraphics[width=0.49\textwidth]{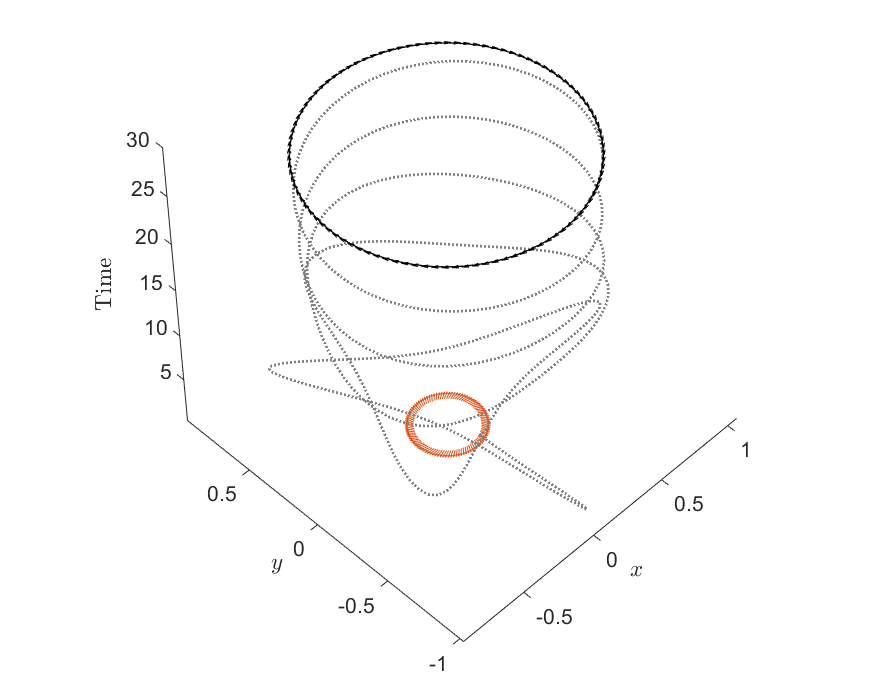}\hfill
		\includegraphics[width=0.49\textwidth]{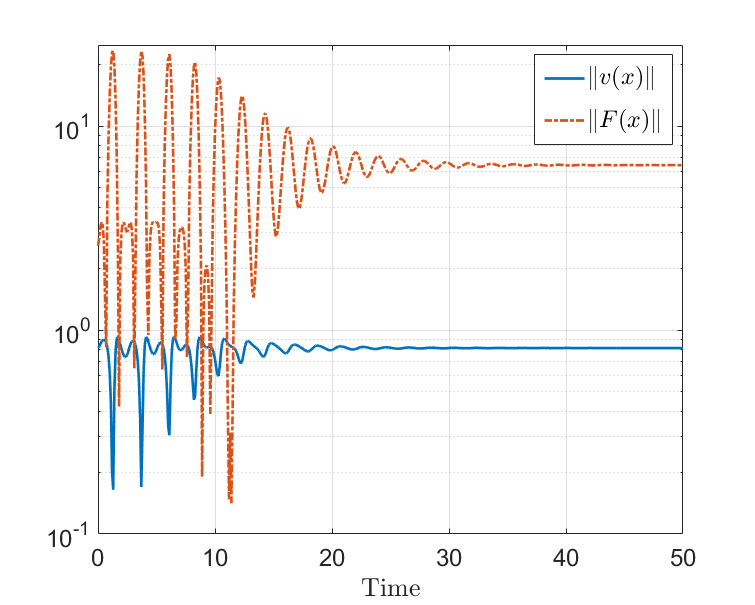}\\
		\includegraphics[width=0.49\textwidth]{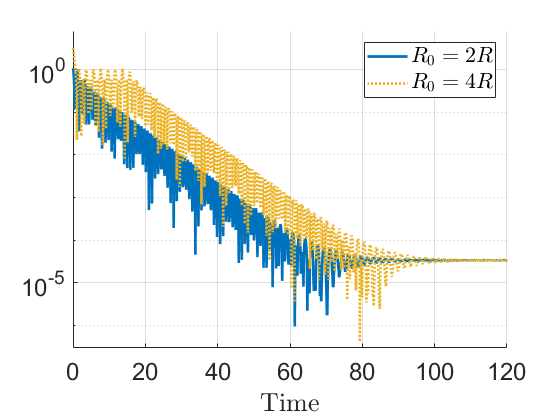}\hfill
		\includegraphics[width=0.49\textwidth]{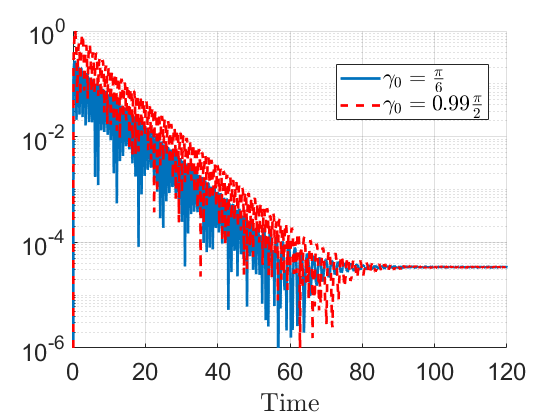}
		\caption{Mill ring stability in the power-law potential \eqref{plaw} with $a=4,b=1,\alpha=10, \beta=3$, $N=200$. Top left: an initial ring configuration converges to mill ring solution (sample agent trajectory in grey). Top right: energy evolution of the swarm towards the mill. Bottom left: evolution of the distance to the stable mill radius $R$ for mill configurations departing from an initial radius $R_0$. Bottom right: evolution of the distance to stable mill radius for rings of radius $R$, initial velocity rotated from the tangential velocity by $\gamma_0$. The evolution to the stable mill configuration is robust to perturbations.}
		\label{mill1}
	\end{figure}
	
\subsection{Controlling mills for a power-law potential}
Although the local stability analysis is enough for our purposes of controlling the system \eqref{syst} towards mill ring solutions, we observe that for power-law potentials \eqref{plaw}, the mill ring is globally asymptotically stable for solutions with rotational symmetry.  Figure \ref{mill1} (top row) shows the evolution of a particular solution, not a small perturbation, with rotational symmetry, converging towards a stable mill with radius given by \eqref{e-radius}, similar to a nonlinear damped oscillator. The second row further illustrates the stability of the mill ring by considering the evolution of two types of perturbations. On the left, the initial configuration is a mill ring with radius different from the stable solution. On the right, the initial configuration is a ring of stable radius, however the tangential velocities are shifted by an angle $\gamma_0$. In both plots, the vertical axis represents the distance with respect to the stable radius, and we can observe that for both types of perturbations the uncontrolled dynamics stabilize towards the mill.   
		
Feedback controls can be used to induce or accelerate convergence to mill ring solutions. For example, given a stable flock ring configuration, the system can be stabilized towards a mill through the action of the feedback law
\begin{equation*}
	u_i(v_i(t))=-M\left(v_i-\sqrt{\frac{\alpha}{\beta}}\frac{x_i^\perp}{|x_i^\perp|}\right)\,.
\end{equation*}
A design alternative is to resort to instantaneous controls \cite{APZ14}, which can be interpreted as feedback laws in the same spirit of model predictive control strategies. We synthesize a feedback control by solving
\begin{equation}\label{instfeed}
	\underset{u\in[-1,1]^2}{\min} \sum\limits_{i=1}^N \left|v_i-\sqrt{\frac{\alpha}{\beta}}\frac{x_i^\perp}{|x_i^{\perp}|^2}\right|+(|x_i-x_m|^2-R_m^2)^2+\lambda_1|u|+\lambda_2|u|^2\,,\quad \lambda_1,\lambda_2>0\,,
\end{equation} 
where $x_m$ corresponds to the center of mass of the swarm, $R_m$ is the desired mill radius, and $(x_i,v_i)$ is the future state of the system after a small control horizon $\Delta t$. We consider $\ell_1$ and $\ell_2$-norm control penalties to induce sparsification in time. For the sake of real-time computability, this optimization is reduced to a single control signal $u\in\R^2$, which enters the dynamics through an incremental rotation of $\frac{2\pi}{N}$. This control differs from the signal that would be obtained optimizing each $u_i$ separately, however it still succeeds in stabilizing around the mill ring solution, as shown in Figure \ref{instcont}, presumably due to the large basin of attraction surrounding the mill ring solution. 
\begin{figure}[!ht]
		\centering
		\includegraphics[width=0.33\textwidth]{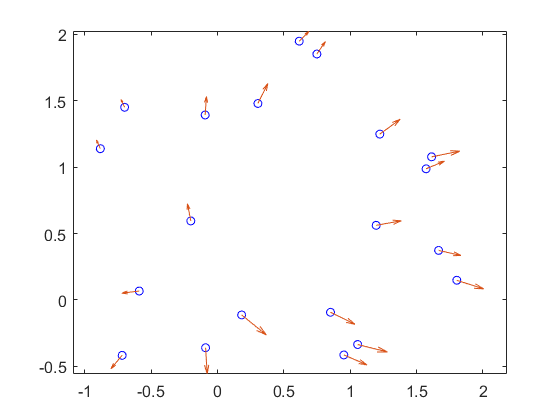}\hfill
		\includegraphics[width=0.33\textwidth]{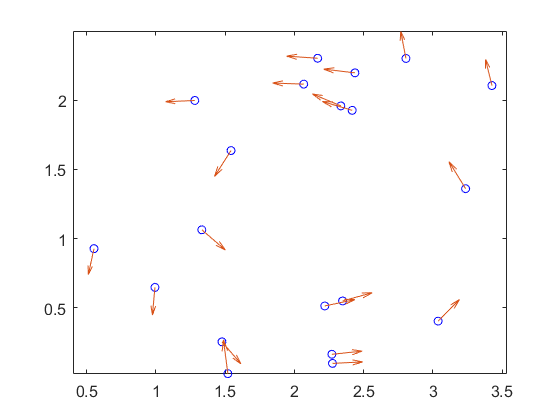}\hfill
		\includegraphics[width=0.33\textwidth]{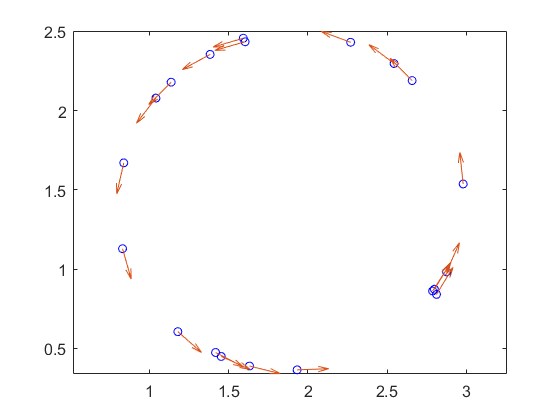}\\
		\includegraphics[width=0.49\textwidth]{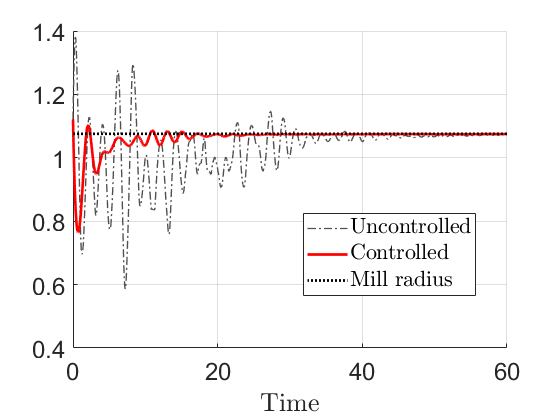}\hfill
		\includegraphics[width=0.49\textwidth]{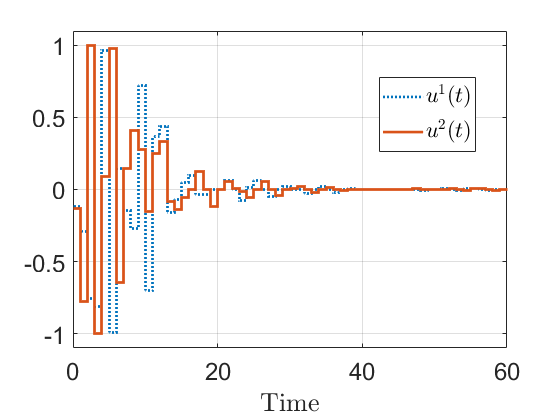}\\
		\caption{Stabilization towards a mill ring for the power-law potential \eqref{plaw} with $a=4,b=1$, $\alpha=10, \beta=3$, $N=20$, using the feedback \eqref{instfeed}. Top (left to right): swarm at $t=0,10,40$. Bottom left: evolution of the configuration radius towards the stable mill.  Bottom right: control signal.}\label{instcont}
\end{figure}

We apply a similar idea to control a stable mill towards a flocking configuration of different radius. In this case we compute one control variable per agent by solving 	
\begin{equation}\label{instfeed2}
	\underset{u\in[-1,1]^{2N}}{\min} \sum\limits_{i=1}^N |v_i-\bar v|^2+(|x_i-x_m|^2-R_f^2)^2+\lambda|u_i|^2\,,\quad \lambda>0\,,
\end{equation} 
where $\bar v,R_f$ are the desired flocking velocity and radius, respectively. Figure \ref{mill2flock} illustrates the transition from the milling to the flocking regime. 	
		\begin{figure}[!ht]
		\centering
		\includegraphics[width=0.33\textwidth]{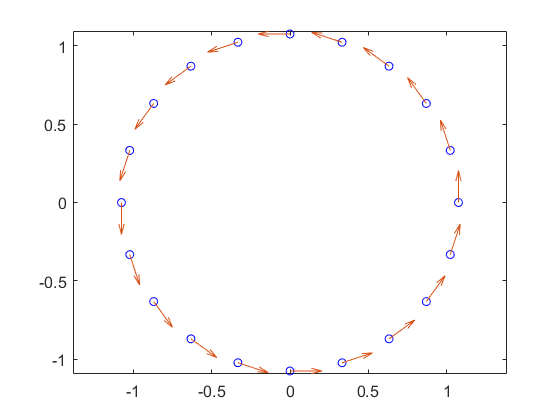}\hfill
		\includegraphics[width=0.33\textwidth]{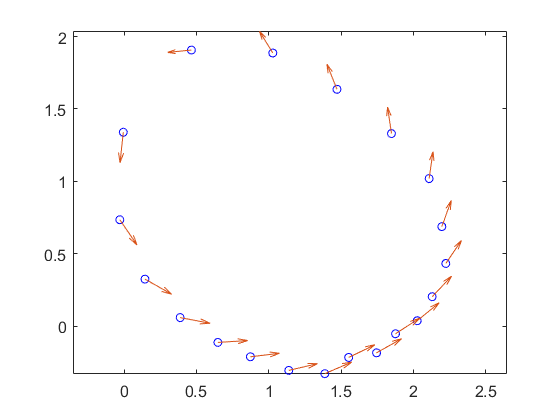}\hfill
		\includegraphics[width=0.33\textwidth]{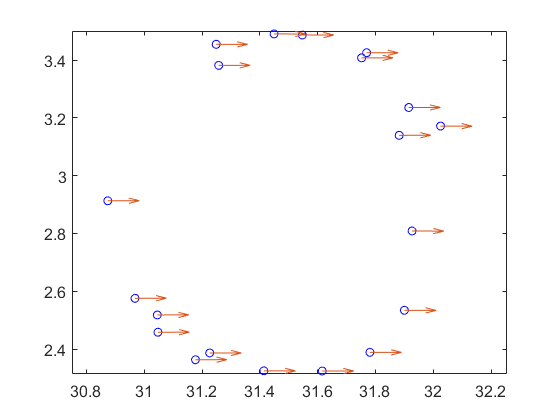}\\
		\includegraphics[width=0.49\textwidth]{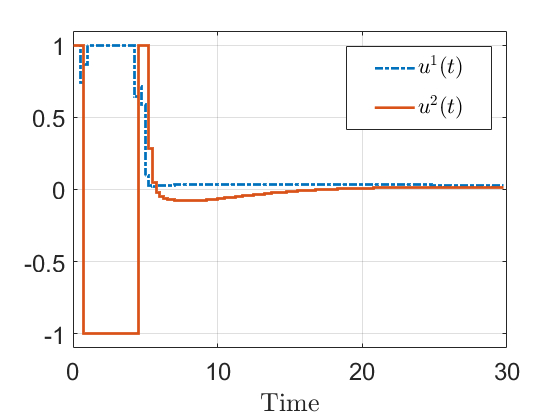}\hfill
		\includegraphics[width=0.49\textwidth]{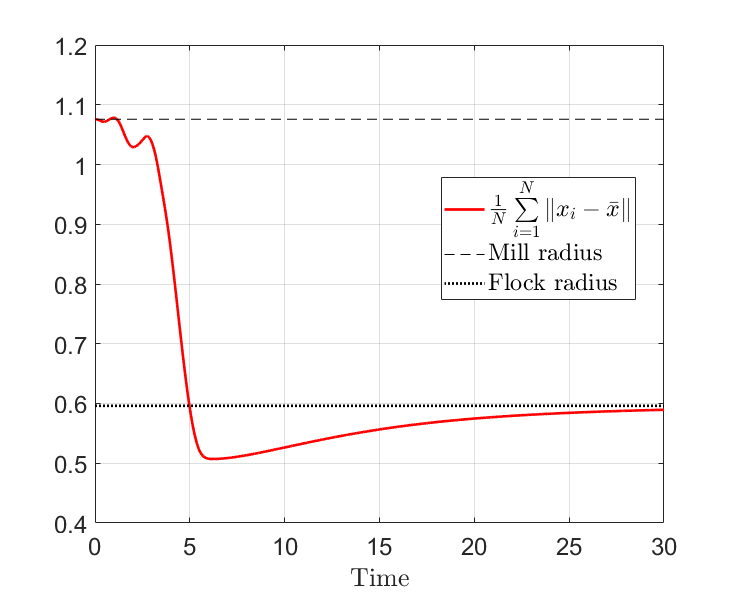}
		\caption{Controlled transition from mill to flock for the power-law potential ($a=4,b=1,\alpha=10,\beta=3,N=20$), using the feedback \eqref{instfeed2}. Top: swarm at $t=0,4,40$. Bottom left: a sample control signal. Bottom right: evolution of the swarm radius, from a stable mill to a stable flock. }\label{mill2flock}
	\end{figure}
	
	
	\section{Proof of Theorem \ref{mainthm}}\label{sec_proof_mainthm}
	This section is devoted to provide the full detail of the proof of Theorem \ref{mainthm}, sketched in Section \ref{sec_bounded_interactions}.
	
	\subsection{Proof of Step 1}\label{sec_proof_step1}
	In this first step, the goal is to steer the control system from any initial point to a point in $\Omega_\eps$. 
	
	\subsubsection*{Step 1.1. Jurdjevic-Quinn stabilization}
	Let $
		\gamma>\max\left(1,\frac1M \sqrt{\frac{\alpha^3}{\beta}}\right)
	$
	be fixed. We apply the feedback control  
	\begin{equation}
		u_i(v_i)=
		\begin{cases}
			0& \text{if }\, |v_i|\geq 2\gamma\sqrt{\frac{\alpha}{\beta}},\\
			-M\frac{v_i}{|v_i|}\left(2-\frac{|v_i|}{\gamma\sqrt{\frac{\alpha}{\beta}}}\right)& \text{if }\, |v_i|\in\left(\gamma\sqrt{\frac{\alpha}{\beta}},2\gamma\sqrt{\frac{\alpha}{\beta}}\right),\\
			-M\frac{v_i}{|v_i|}& \text{if }\, |v_i|\in\left[\frac{ \sqrt{\frac{\alpha}{\beta}}}{\gamma},\gamma\sqrt{\frac{\alpha}{\beta}}\right],\\
			-M\frac{\gamma}{\sqrt{\frac{\alpha}{\beta}}} v_i & \text{if }\, |v_i|<\frac{ \sqrt{\frac{\alpha}{\beta}}}{\gamma}.
		\end{cases}\label{e-controlO}
	\end{equation}
%
	Since the control is Lipschitz with respect to the $(x,v)$ variables, we have existence and uniqueness of solutions of \eqref{syst} for a fixed initial condition. Note that the control law satisfies the constraint $|u_i|\leq M$ by construction. Setting the total energy of the system as
	\begin{equation*}
		V(t)=\frac12 \sum_{i=1}^N |v_i(t)|^2+{\frac1{2N}} \sum_{\substack{i,j=1\\ i\neq j}}^N W(x_i(t)-x_j(t)) ,
	\end{equation*}
	using \eqref{syst}, we have
	\begin{equation*}
			\dot V =\sum_{i=1}^N v_i\cdot \dot v_i+{\frac1{2N}} \sum_{\substack{i,j=1\\ i\neq j}}^N \nabla W(x_i-x_j)(\dot x_i-\dot x_j) =
			\sum_{i=1}^N \left((\alpha-\beta|v_i|^2)|v_i|^2+v_i\cdot u_i\right).
	\end{equation*}
	Notice that by skew-symmetry of $\nabla W(x_i-x_j)$, we have
	$$
	\sum_{\substack{i,j=1\\ i\neq j}}^N \nabla W(x_i-x_j)\cdot v_j=-\sum_{\substack{i,j=1\\ i\neq j}}^N \nabla W(x_i-x_j)\cdot v_i\,.
	$$ 
	Given $a_1=\frac{1}{\gamma} \sqrt{\frac{\alpha}{\beta}}$ and $a_2=\gamma\sqrt{\frac{\alpha}{\beta}}$, we split the indices $i=1,\ldots, N$ in the sets 
	$$
		I_{\infty}(t)=\left\{i:  |v_i(t)|>a_2\right\},\, I_{1}(t)=\left\{i: |v_i(t)|\in\left[a_1,a_2\right]\right\}\mbox{ and}\,
		I_{0}(t)=\left\{i:|v_i(t)|<a_1\right\}.
	$$
	Apply the control law \eqref{e-controlO} and notice that for every $i\in I_\infty(t)$ we obtain 
$$
	\left\{v_i\cdot u_i=0 \quad\textrm{or} \quad v_i\cdot u_i=-M|v_i|\left(2-\frac{|v_i|}{\gamma\sqrt{\frac{\alpha}{\beta}}}\right)<0\right\} \quad\mbox{and}\quad (\alpha-\beta|v_i|^2)|v_i|^2<0.
$$
	We deduce
	\begin{multline*}
	\dot V \leq \sum_{i\in I_\infty(t)}(\alpha-\beta|v_i|^2)|v_i|^2+\sum_{i\in I_1(t)}|v_i|(\alpha|v_i|-\beta|v_i|^3-M)\\
	+\sum_{i\in I_0(t)}\left(\alpha-\beta|v_i|^2-M\frac{\gamma}{\sqrt{\frac{\alpha}{\beta}}}\right)|v_i|^2.
	\end{multline*}
	The maximum of the function $v\to \alpha v-\beta v^3$ over $v\geq 0$ is $\albe<M$. We have $\alpha|v_i|-\beta|v_i|^3-M<0$ and $\alpha-M\frac{\gamma}{\sqrt{\frac{\alpha}{\beta}}}<0$ by the choice of $\gamma$. This implies that
	\begin{equation}\label{e-stimaV}
		\dot V\leq \sum_{i\in I_\infty(t)}(\alpha-\beta|v_i|^2)|v_i|^2-\sum_{i\in I_1(t)\cup I_0(t)}\beta|v_i|^4.
	\end{equation}
	The right-hand side is then nonpositive being the sum of two nonpositive terms, hence $\dot V\leq 0$. 
	Note that we cannot directly apply the LaSalle invariance principle to the system, because we cannot ensure boundedness of the trajectories and the functional $V$ is not proper.
	
	Nevertheless, $V$ is bounded below, since both $|v|$ and $W$ are bounded below, as a consequence of the corresponding assumption on $U$. Since $V(t)$ is decreasing, hence bounded above, both $v_i(t)$ and $W(x_i(t)-x_j(t))$ are bounded. Thanks to the assumption \ref{U1}, $\nabla W$ is bounded. This implies that $\dot v_i(t)$ is bounded too.
	
	\begin{lemma}\label{l-v0}
		{The system \eqref{syst} with control law \eqref{e-controlO} satisfies}	$\lim_{t\to +\infty} v_i(t)=0$ for $i=1,\ldots, N$. 
	\end{lemma}
	
	\begin{proof}
		By contradiction, if this is not the case, there exists an index $i$ and a sequence of times $t_k\to\infty$ such that $|v_i(t_k)|>C$. Since $\dot v_i$ is bounded, this implies that there exists a uniform $\tau$ such that $|v_i(t)|>C/2$ for all $t\in(t_k-\tau,t_k+\tau)$. By using this property in \eqref{e-stimaV} discarding the first term in the right-hand side, we infer that $V(t)\to -\infty$. This raises a contradiction.
	\end{proof}
	
	\begin{lemma}\label{l-F0}
			{The system \eqref{syst} with control law \eqref{e-controlO} satisfies}
		$\lim_{t\to +\infty} F_i(x(t))=0$ for $i=1,\ldots, N$. 
	\end{lemma}
	
	\begin{proof}
		Since the $v_i$s are bounded, then the functions $x_i(t)-x_j(t)$ are Lipschitz. Since  \ref{U1} holds, then both $\nabla W(x_i(t)-x_j(t))$ and its derivative are bounded; hence functions $\nabla W(x_i(t)-x_j(t))$ are Lipschitz too, with a Lipschitz constant that we denote with $L$. 
		
		Assume now, by contradiction, that there exists an index $i$ such that $F_i(x(t))$ does not converge to 0. Thus, there exists a sequence of times $t_k\to\infty$ such that 
		\begin{enumerate}
			\item either $F_i(x(t_k))\to \bar F$ for some non-zero vector $\bar F$;
			\item or $|F_i(x(t_k))|\to +\infty$.
		\end{enumerate}
		
		In the first case, for each $\eps>0$ there exists an index $K>0$ such that $\|F_i(x(t))-\bar F\|<\eps+L(t-t_k)$ for all $k>K$ and $t>t_k$. Recalling that $\lim_{t\to +\infty} v_i(t)=0$, take now $\eta>0$ sufficiently small and $k$ sufficiently large to have both $|v_i(t)|<\eta$ and $|\al v_i(t)-\beta v_i(t)|\cdot |v_i(t)|^2<\eta$ for all $t>t_k$. This implies $|\dot v_i(t_k+\tau)-\bar F|<2\eta+L\tau$ for all $\tau>0$. Since $|v_i(t_k)|<\eta$, then $|v_i(t_k+\tau)-\bar F\tau|<\eta+2\eta\tau+L\frac{\tau^2}2$ for $\tau>0$. Fix $\tau>0$ sufficiently small to have $L\frac{\tau^2}2<|\bar F|\frac{\tau}2$ and observe that this implies $|v_i(t_k+\tau)|>|\bar F|\frac{\tau}2-\eta-2\eta\tau$. Let $\eta\to 0$ and note that this implies $\lim_{k\to\infty}v_i(t_k+\tau)\neq 0$. This raises a contradiction.
		
		The second case is similar: consider the unit vectors $\frac{F_i(x(t))}{|F_i(x(t))|}$, that admit a converging subsequence (that we do not relabel) to an unitary vector $\bar F$. Following computations of the previous case, we have $|v_i(t_k+\tau)|>|F_i(x(t_k))| \frac{\tau}2-\eta-2\eta\tau$, which does not converge to $0$ for $\eta\to 0$ and $k\to\infty$. This raises a contradiction. 
	\end{proof}
	
	Finally, let us choose a time $T_0$ at which we stop the control strategy \eqref{e-controlO}. This choice is driven by correctly initializing the next step. Let $\eps'>0$ be a constant to be chosen later. Since both $v_i(t)$ and $F_i(x(t))$ converge to $0$, we choose a time $T_0$ at which $|v_i(T_0)|<\eps'$ and $|F_i(x(T_0))|<\eps'$ for $i=1,\ldots,N$.

	\subsubsection*{Step 1.2. Reaching $\Omega_\eps$ in finite time}
	We now steer each $v_i$ exactly to zero. {We first define the trajectory for each $v_i$: this in turn gives the trajectory of the $x_i$ by integration, and the control $u_i$ by identification in the second equation of \eqref{syst}}. Choose $T_{1,i}=T_0+\tfrac1{\eps'}|v_i(T_0)|$ and define
	$$
	v_i(t)=\begin{cases}
		v_i(T_0)-\eps'(t-T_0)\frac{v_i(T_0)}{|v_i(T_0)|}&\mbox{~~for~}t\in[T_0,T_{1,i}],\\
		0&\mbox{~~for~}t>T_{1,i}.
	\end{cases}
	$$
	Then choose $T_1=\max(T_{1,i})$ as the final time of the strategy.
	A direct computation shows that 
	\begin{equation}\label{e-xi12}
		|x_i(t)-x_i(T_0)|\leq \int_{T_0}^{T_{i,1}}|v_i(T_0)|-\eps'(t-T_0)\,dt\leq \frac{|v_i(T_0)|^2}{2\eps'}<\frac{\eps'}{2}
	\end{equation}
	for every $t\in[T_0,T_1]$. 
	{Since $\grad W(x_i-x_j)$ is $L$-Lipschitz continuous and bounded, as recalled in Lemma \ref{l-F0}, we infer that 
	$
	\|F_i(T_0+t)\|\leq \|F_i(T_0)\|+L\tfrac{\eps'}{2}<(1+\tfrac{L}2)\eps'.
	$
	Notice that $u$ bounded implies that all solutions of (1.1) are Lipschitz with respect to time. This implies that there exists $\bar x$ (that we cannot choose) such that} $(x(T_1),v(T_1))=(\bar x,0)\in\Omega_\eps$ by imposing $\eps'\leq \eps/(1+\tfrac{L}2)$.
	By a simple estimate in the second equation of \eqref{syst}, the control satisfies
	$$
	|u_i|\leq |\dot v_i|+(\alpha-\beta |v_i|^2)|v_i|+|F_i|< \eps'+\alpha \eps'+(1+\tfrac{L}2)\eps'\leq M,  
	$$
	by imposing $\eps'\leq M/(2+\alpha+\tfrac{L}2)$.
	Summing up, choosing $\eps'=\min\left(\frac\eps{1+\tfrac{L}2},\frac{M}{2+\alpha+\tfrac{L}2}\right)$, all conditions are satisfied.

	\subsection{Proof of Step 2}\label{sec_proof_step2}
	Step 2 is obvious, the control system \eqref{very_simple_control_system} being straightforward to control. We do not provide any detail. \FR{Given a connected neighborhood $\mathcal{N}$ of $(\bar x,0)$ inside $\Omega_\eps$, we are then able to steer the system to any chosen point in $\mathcal{N}$.}
	
	\subsection{Proof of Step 3}\label{sec_proof_step3}
	
	\subsubsection*{Step 3.1: Reaching flocks}
	{Fix a unit vector $\bar v$ and note that $\Omega_\eps$ open implies that there exists $\delta>0$ such that the configuration $(\bar x, \delta \bar v)$ belongs to the neighborhood $\mathcal{N}$ of $(\bar x, 0)$ given at Step 2}. Then, we can steer the system from $(\bar x,0)$ to $(\bar x, \delta \bar v)$ at a time $T_2>T_1$ with a control $u=\bar u+z$ satisfying $\|u\|<2\eps<M$, again by a local controllability argument.
	
	We then choose the controls $u_i=F_i$ on the time interval $[T_2,+\infty)$. The velocity variables are then  the solutions of 
	\begin{equation}\label{e-vfree}
		\dot v_i(t)=(\alpha -\beta|v_i(t)|^2)v_i(t),\qquad v_i(T_2)= \tfrac{\delta}2 \bar v.
	\end{equation}
	They all coincide at each time, i.e., $v_i(t)=v_j(t)$, hence relative positions are all constant with respect to time, i.e $x_i(t)-x_j(t)=\bar x_i-\bar x_j$. This in turn implies that interaction forces keep being constant with respect to time, thus $\|u_i(t)\|=\|F_i(t)\|=\|F_i(T_2)\|<\eps$, since $(x(T_2),v(T_2))\in\Omega_\eps$.
	
	Moreover, all velocities converge to $\sqrt\frac\alpha\beta \bar v$ for $t\to +\infty$, since they solve \eqref{e-vfree}. This implies that the system converges to an $\eps$-flock, as stated.
	
	\begin{remark}
		The motion of the velocity variables $v_i(t)$, solutions of \eqref{e-vfree}, exactly follows \emph{heteroclinic trajectories}, in the sense that $v_i(t)$ passes from (a neighborhood of) the unstable equilibrium $0$ to the asymptotically stable family of equilibria $\left\{\|v\|=\sqrt\frac\alpha\beta\right\}$. The existence of such heteroclinic trajectories is certainly one of the main interesting features of the dynamics of \eqref{syst}, promoting convergence to flocks or mills.
	\end{remark}

	\subsubsection*{Step 3.2: Passing from a flock to another flock}
	Assume that the system is at (or near) a flock of velocity $\bar v^0$. We want to steer the system to another flock, of velocity $\bar v^1$. Along the motion, the relative positions $x_i-x_j$ will remain constant.
	
	The strategy that we use here is by \emph{quasi-static deformation}. Take a continuous path $\tau\in[0,1]\mapsto\bar v(\tau)$ such that $\bar v(0)=\bar v^0$ and $\bar v(1)=\bar v^1$, satisfying $\|\bar v(\tau)\|=\sqrt\frac\alpha\beta$ for every $\tau\in[0,1]$, e.g., the shortest arc on the circle. 
{Given any fixed $\tau\in[0,1]$, the flock under consideration is $(\bar x(t) = \bar x(0)+\bar v(\tau) t, \bar v(\tau))$.}
Since relative positions do not change, forces $F_i$ do not change and then, for each $\tau$, we have a flock of velocity $\bar v(\tau)$.	
	Of course, the corresponding path of flocks, parametrized by $\tau\in[0,1]$ is \emph{not} a solution of \eqref{syst}. It is rather to be thought of as a path of equilibrium points for the dynamics \eqref{syst}.
	Following the idea of \cite{coron-trelat04,coron-trelat06}, we track this path, in large time, by designing appropriate feedback controls. 
{To this aim, for any given fixed $\tau\in[0,1]$, we linearize the control system \eqref{syst} at the corresponding flock: we set
$$
x_i(t) = \bar x_i(0)+\bar v(\tau)t+\delta x_i(t), \qquad v_i(t)=\bar v(\tau)+\delta v(t) .
$$
Plugging in \eqref{syst}, using that $F_i(\bar x(t))=0$ and that $\vert\bar v(\tau)\vert=\sqrt{\frac{\alpha}{\beta}}$, we get, at the first order, the linear system
$$
\delta\dot x_i(t) = \delta v_i(t),\qquad \delta \dot v_i(t) = 2\langle\bar v(\tau),\delta v(t)\rangle\bar v(\tau)+dF_i(\bar x_i(0)+\bar v(\tau)t).\delta x(t)+u_i(t) .
$$
We make a change of control by setting $\delta u_i(t) = -2\langle\bar v(\tau),\delta v(t)\rangle\bar v(\tau)-dF_i(\bar x_i(0)+\bar v(\tau)t).\delta x(t)+w_i(t)$, thus obtaining the very simple control system
$$
\delta\dot x_i(t) = \delta v_i(t),\qquad \delta \dot v_i(t) = w_i(t),
$$
i.e., we recover the system \eqref{very_simple_control_system}.
}

{This has been done for $\tau$ fixed. Now, the idea is to perform the above deformation slowly in time, by setting $\tau=\varepsilon t$, for some $\varepsilon>0$ small enough, and thus $t\in[0,1/\varepsilon]$, and compensate for the errors by designing an adequate feedback control.}

Of course, for every fixed value of $\tau$, the above control system is linear autonomous, of the form $\dot X(t) = AX(t)+Bu(t)$ for some matrices $A$ and $B$. It obviously satisfies the {Kalman controllability condition} and is thus controllable and also feedback stabilizable (for instance by standard pole shifting, see, e.g., \cite{LeeMarkus,Sontag,trelatbook2005}). But now, along the path of flocks {that we want to track slowly in time}, we do not have anymore a linear autonomous control system, but a linear \emph{instationary} control system, of the form $\dot X(t) = A(\varepsilon t)X(t)+B(\varepsilon t)u(t)$ for some matrices $A(\varepsilon t)$ and $B(\varepsilon t)$ depending on time {but varying slowly in time}. For every $\tau$, the pair $(A(\tau),B(\tau))$ still satisfies the Kalman condition. For linear instationary control systems the Kalman condition is not sufficient to ensure controllability nor stabilizability properties (see counterexamples, e.g., in \cite{khalil,Sontag,trelatbook2005}). But here, we follow the path slowly in time: by setting $\tau=\varepsilon t$, the abovementioned control system takes the form $\dot X(t) = A(\varepsilon t)X(t)+B(\varepsilon t)u(t)$ and is therefore a  \emph{slowly-varying} (in time) linear control system, satisfying the Kalman condition. As explained in detail in \cite{coron-trelat04}, and according to an argument of \cite[Chap. 9.6]{khalil}, if $\varepsilon>0$ is small enough then the Kalman condition is still sufficient to ensure that this slowly-varying linear control system can be feedback stabilized by usual pole shifting, with a feedback control of the form $u(t) = K(\varepsilon t) X(t)$. Note anyway that such a feedback is also slowly varing in time, so is not a ``pure" feedback. One may want to obtain a feedback, not depending on time, but defined piecewise in time. This is possible by slightly modifying the above definition of the feedback. The resulting \emph{staircase method} has been used, e.g., in \cite{PoucholTrelatZuazua}.
	
	Eventually, such feedback controls make it possible to track the path of flocks and thus steer the control system, in time $1/\varepsilon$, to any point of a neighborhood of $\bar v_1$. 
	
	If one moreover aims to choose precise $x$-positions (keeping anyway the same relative positions as those of the initial flock), it is sufficient to observe that all such configurations differ from a translation vector $X$. Therefore, it suffices to use a quasi-static deformation on the positions as well.

	\subsubsection*{Proof of Steps 3.3, 3.4, 3.5: reaching mills, flock rings}
	The strategy, described in Section \ref{sec_bounded_interactions}, is similar to what has been described above, and we thus do not give any detail.

	\section{Proof of Theorem \ref{mainthm_variant}}\label{sec_proof_thm2}
	
	\subsection{Proof of the first statement} \label{s-21}
	The proof of the first statement is almost identical to the proof of Theorem \ref{mainthm}, with the following differences:
	\begin{itemize}
		\item In Step 1.1, we follow the proof until $\dot V\leq 0$, due to \eqref{e-stimaV}. Since $V$ is bounded below, both $v_i$ and $W_i(x_i-x_j)$ are bounded below. Now, using the assumptions \ref{U2}, boundedness of $W$ implies that $x_i(t)-x_j(t)$ is bounded away from $0$ by a constant $2\ell>0$, hence $\nabla W( x_i(t)-x_j(t))$ is bounded. This in turn implies that $\dot v_i(t)$ is bounded as well. One can then prove Lemmas \ref{l-v0} and \ref{l-F0} in this case too.
		\item In Step 1.2, we observe that \eqref{e-xi12} ensures that $|x_i(t)-x_j(t)|>\ell$, provided $\eps'<\sqrt{2\ell}$. We then use \ref{U2} to ensure that $\nabla W$ is $L$-Lipschitz continuous for $|x_i(t)-x_j(t)|>\ell${, for some $L>0$}. We now choose $\eps'=\min(\sqrt{2\ell},\tfrac{2}{L},\tfrac\eps2)$ to ensure that $(x(T_1),v(T_1))\in\Omega_\eps$.
	\end{itemize}
	
	\subsection{Proof of the second statement}\label{s-22}
	We follow the sketch of the proof given in Section \ref{sec_unbounded_interactions}.
	
	\subsubsection{Proof of Step 1} 
	
	\paragraph{\bf Step 1.1: Fictitious purely radial potential} 
	The proof is based on the method of ``artificial potential field'', which is widely used in robotics (see, e.g., \cite[Chap. 7]{spong}). Replace the  potential $U$ with a purely repulsive potential $\tilde U$, that is \FR{chosen as follows. Define $\eta=\frac14(M-(M_{\alpha,\beta}+\tilde M_F))>0$ and choose $R_0>0$ such that $|U'(r)|<\eta$ for all $r\geq R_0$. This is possible because $\lim_{r\to+\infty}U'(r)=0$.  Take now $\phi:(0,+\infty)\to\R$ a $C^\infty$ function with bounded $C^1$ derivative, satisfying}
	\FR{\begin{itemize}
	\item $\phi(r)=1$ for $r\in (0,R_0]$ and $\phi(r)=0$ for $r\in[R_0+1,+\infty)$;
	\item $\phi$ is decreasing for $r\in[R_0,R_0+1]$.
	\end{itemize}
	Define 
	$$\tilde U'(r)=\phi(r)(U'(r)-\tilde M_F)-\eta\frac{1}{1+r^2}$$}
	\FR{and $\tilde U(r)=\int_1^r \tilde U'(r)$. Then $\tilde U$ is $C^2$, satisfies $\lim_{r\to+\infty} \tilde U'(r)=0$ and \ref{U2}. More crucially, $\tilde U'(r)<0$ for all times, i.e., $\tilde U$ is purely repulsive. Moreover, $|\tilde U'(r)-U'(r)|\leq \tilde M_F+\eta$ for $r\in (0,R_0)$, and $|\tilde U'(r)-U'(r)|\leq |\phi(r)-1| |U'(r)|+\phi(r)\tilde M_F+\eta\leq \eta+\tilde M_F+\eta$ for $r\in [R_0,+\infty)$. Then
	\begin{equation}\label{e-U'}
	|\tilde U'(r)-U'(r)|\leq \tilde M_F+2\eta.
	\end{equation}}
	\FR{Define now $\tilde W(x)=\tilde U(|r|)$ and choose the control $u_i=F_i(x)-\tilde F_i(x) +w_i$ (see \eqref{control_fictitious}), so that the new control system is \eqref{syst2}, with the new controls $w_i$ satisfying $\|w_i\|\leq \tilde M$ with $\tilde M$ defined by \eqref{tildeM}. Note that \eqref{e-U'} implies 
	\begin{equation}\label{e-w-const}
	\tilde M\geq M-(\tilde M_F-2\eta)=\tilde M_{\alpha,\beta}+\eta	,
	\end{equation}
		i.e., the system \eqref{syst2} satisfies the assumptions of the first statement of Theorem \ref{mainthm}. This will be used in the next step.}
	
	\paragraph{\bf Step 1.2: Blowing-up all agents} 
	Fix $\eps>0$ to be chosen later. Applying Step 1 of Theorem \ref{mainthm} to the control system \eqref{syst2}, we can steer it to $\tilde \Omega_\eps = \{ (x,v)\ \mid\ v=0,\ \|\tilde F(x)\|<\eps \}$ in finite time $T_0$.  The crucial observation here is that \FR{$\tilde U'(r)\leq -\eta\frac{1}{1+r^2}$}, i.e., all forces are purely repulsive. We now study the configuration $(x(T_0),v(T_0))=(\bar x,0)\in \tilde \Omega_\eps$. Since it is a configuration of $N$ agents in the plane, the convex closure of positions $(x_1,\ldots, x_N)$ is a polygon of $n\leq N$ vertices, in which at least one of the internal angles is smaller than $\frac{n-2}{n}\pi$, thus smaller than $\frac{N-2}{N}\pi$.
	
	By relabelling indices, we assume that $x_N$ is one of those vertices. By a simple geometrical observation, all interaction forces point outwards of the polygon (see Figure \ref{f-angle1}). 
	\begin{figure}[h]
		\centerline{\tikzset{every picture/.style={line width=0.65pt}} 

\begin{tikzpicture}[x=0.6pt,y=0.6pt,yscale=-1,xscale=1]

\draw [line width=1.5]    (290,110) -- (200,180) ;
\draw [line width=1.5]    (290,110) -- (380,180) ;
\draw  [fill={rgb, 255:red, 153; green, 153; blue, 153 }  ,fill opacity=1 ] (320.78,134.57) .. controls (313.44,143.06) and (302.38,148.46) .. (290,148.46) .. controls (277.73,148.46) and (266.75,143.15) .. (259.41,134.79) -- (290,110) -- cycle ;
\draw    (291,110) -- (290.04,62) ;
\draw [shift={(290,60)}, rotate = 448.85] [color={rgb, 255:red, 0; green, 0; blue, 0 }  ][line width=0.75]    (10.93,-3.29) .. controls (6.95,-1.4) and (3.31,-0.3) .. (0,0) .. controls (3.31,0.3) and (6.95,1.4) .. (10.93,3.29)   ;
\draw   (291,110) -- (381,180) -- (351,250) -- (241,240) -- (201,180) -- cycle ;
\draw  [dash pattern={on 0.84pt off 2.51pt}]  (291,110) -- (338.29,81.04) ;
\draw [shift={(340,80)}, rotate = 508.52] [color={rgb, 255:red, 0; green, 0; blue, 0 }  ][line width=0.75]    (10.93,-3.29) .. controls (6.95,-1.4) and (3.31,-0.3) .. (0,0) .. controls (3.31,0.3) and (6.95,1.4) .. (10.93,3.29)   ;
\draw  [dash pattern={on 0.84pt off 2.51pt}]  (290,110) -- (251.79,90.89) ;
\draw [shift={(250,90)}, rotate = 386.57] [color={rgb, 255:red, 0; green, 0; blue, 0 }  ][line width=0.75]    (10.93,-3.29) .. controls (6.95,-1.4) and (3.31,-0.3) .. (0,0) .. controls (3.31,0.3) and (6.95,1.4) .. (10.93,3.29)   ;
\draw  [dash pattern={on 0.84pt off 2.51pt}]  (241,240) -- (299.31,81.88) ;
\draw [shift={(300,80)}, rotate = 470.24] [color={rgb, 255:red, 0; green, 0; blue, 0 }  ][line width=0.75]    (10.93,-3.29) .. controls (6.95,-1.4) and (3.31,-0.3) .. (0,0) .. controls (3.31,0.3) and (6.95,1.4) .. (10.93,3.29)   ;
\draw  [dash pattern={on 0.84pt off 2.51pt}]  (351,250) -- (280.77,81.85) ;
\draw [shift={(280,80)}, rotate = 427.33000000000004] [color={rgb, 255:red, 0; green, 0; blue, 0 }  ][line width=0.75]    (10.93,-3.29) .. controls (6.95,-1.4) and (3.31,-0.3) .. (0,0) .. controls (3.31,0.3) and (6.95,1.4) .. (10.93,3.29)   ;

\draw (246,145) node [anchor=north west][inner sep=0.75pt]    {angle $\leq \frac{N-2}{N} \pi $};
\draw (297,49) node [anchor=north west][inner sep=0.75pt]    {$z_{N}$};
\draw (300,107) node [anchor=north west][inner sep=0.75pt]    {$x_{N}$};

\end{tikzpicture}}
		\caption{The outer vector $z_N$.}\label{f-angle1}
	\end{figure}
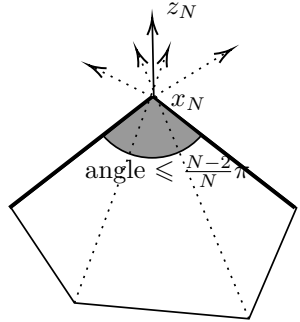
	More precisely, consider $z_N$ to be the unit vector in the direction of the outer angle bisector and note that each component of the force $\tilde F_{Nj}=\frac{1}{N} \nabla \tilde W(\bar x_N-\bar x_j)$ satisfies 
	$$
	\tilde F_{Nj}\cdot z_N\geq \frac1N |\tilde U' (|\bar x_N-\bar x_j|)|\cos(\tfrac{N-2}{2N} \pi).
	$$
	This in turn implies that 
	\begin{eqnarray}\label{e-TF}
	\|\tilde F_N\|&\geq& \tilde F_N\cdot z_N \geq \tfrac1N \sup_{j\neq N} |\tilde U' (|\bar x_N-\bar x_j|)|\cos(\tfrac{N-2}{2N} \pi)\\
	&\geq&\tfrac{1}N \eta \frac{1}{1+\inf_{j\neq N} |\bar x_N-\bar x_j|^2}\cos(\tfrac{N-2}{2N} \pi).\nonumber
	\end{eqnarray}
	\FR{Recall that the original potential $U(r)$ satisfies $\lim_{r\to+\infty}U'(r)=0$. Then, there exists a sufficiently large distance $\mathcal{L}>1$ such that $|U'(r)|<\eta$ for every $r>\mathcal{L}$. Take $\mathcal{L}$ satisfying this condition and define $\eps=\frac1N \eta\frac1{1+\mathcal{L}^2} \cos(\tfrac{N-2}{2N} \pi)$.
	The condition $(\bar x,0)\in \tilde \Omega_\eps$ implies $\|\tilde F_N\|<\eps$. This in turn implies that $|x_N-x_j|>\mathcal{L}$ for all $x\neq N$, due to \eqref{e-TF}}.
	
	Since the distance between $x_N$ and the other agents is larger than $\mathcal{L}$, then the components $\tilde F_{Nj}$ and $\tilde F_{jN}$ of the forces are \FR{smaller than $\eps$}. Hence, all configurations $((\bar x_1,\ldots, \bar x_{N-1},\bar x_N +\tau_N z_N),0)$ with $\tau_N>0$ belong to $\tilde \Omega_{\eps}$. By quasi-static deformation, we can steer $\bar x_N$ arbitrarily far from the other $\bar x_i$'s along the direction $z_N$, by choosing $\tau_N$ sufficiently large.
	{In this quasi-static deformation, the path of steady-states is $((\bar x(\tau)= (\bar x_1,\ldots, \bar x_{N-1},\bar x_N +\tau_N z_N), \bar v(\tau)=0)$. Each of them is a flock. For any given fixed $\tau\in[0,1]$, we linearize the control system \eqref{syst} at the corresponding flock: similarly as in Section \ref{sec_proof_step3}, we arrive at a linear system of the form $\delta\dot x_i(t) = \delta v_i(t)$, $\delta \dot v_i(t) = dF_i(\bar x_i(0)+\bar v(\tau)t).\delta x(t)+u_i(t)$, and then of the form \eqref{very_simple_control_system} by changing the control. We do not give more details since the procedure is the same as in Section \ref{sec_proof_step3}.
}
	
	We next apply the same strategy to the remaining $N-1$ agents and we steer one of them away from all others, while keeping $x_N$ \FR{further than $\mathcal{L}$} due to the choice of $\tau_N$. We repeat the procedure to the remaining $N-2$ agents, while keeping both $x_{N-1},x_N$ \FR{further than $\mathcal{L}$}, and so forth. In finite time, we are able to steer all agents to a configuration $(\bar x,0)\in \tilde \Omega_\eps$ with $|\bar x_i-\bar x_j|$ \FR{larger than $\mathcal{L}$}.
	
	\subsubsection{Proof of Step 2} Since all agents are far one from each other, we have \FR{$|\tilde F_i(x)|<N\eps<\eta$} for $i=1,\ldots,N$. \FR{This means that we can again change the control $w_i$ into $z_i=w_i-F_i(x)$, and that the constraint on the control satisfies $|z_i|\leq M-(\tilde M_F-2\eta)-\eta=\tilde M_{\alpha,\beta}+\eta$, due to \eqref{e-w-const}.}
	
	We can now steer all agents to a circular equidistributed configuration of large radius $R$, again by quasi-static deformation, as follows.
	Choose a point $x^*$ of the plane, that does not belong to any of the lines passing through $(\bar x_i,\bar x_j)$ and apply a coordinate translation to have $x^*=0$. {Consider the half-line starting at $0$ and passing through $\bar x_i$, and define $\tilde x_i$ as the point on the half-line at distance  
$
		R>\frac{\alpha}{\beta \eps}
$ from 0}. We want to steer each particle $x_i(t)$ starting at $\bar x_i$ to such $\tilde x_i$. The crucial observation is that each pairwise distance $|x_i(t)-x_j(t)|$ needs to be kept  larger \FR{than $\mathcal{L}$} to ensure that \FR{$|\tilde F_i(x)|<\eta$} along the motion. 
	
	Notice that each angle $\widehat{\bar x_i x^* \bar x_j}$ is nonzero for $i\neq j$ by the choice of $x^*=0$, hence there exists a minimal angle $\theta>0$. Consider now one of the indices $i$ realizing the maximal distance $|\bar x_i|$ (that we assume to be the index $1$) and move it along the quasi-static trajectory $x_1(\tau)=\bar x_1 +\tau(\tilde x_1-\bar x_1)$. Since $\bar x_1$ was chosen to realize the maximal distance, for each $j\neq 1$ the triangle with vertices $0,\bar x_j, x_1(\tau)$ has an internal angle $\alpha_j$ in $\bar x_j$ that is increasing with respect to time, hence the distance $|\bar x_j-x_1(\tau)|$ is increasing too (see Figure \ref{f-angle2}, left).
	
	\begin{figure}[h]
		\centerline{\tikzset{every picture/.style={line width=0.75pt}} 

\begin{tikzpicture}[x=0.75pt,y=0.75pt,yscale=-.9,xscale=.9]

\draw   (190,100) -- (40,100) -- (140,50) -- cycle ;
\draw  [dash pattern={on 0.84pt off 2.51pt}]  (170,100) -- (260,100) ;
\draw  [dash pattern={on 0.84pt off 2.51pt}]  (140,50) -- (220,100) ;
\draw   (486,50) -- (526,120) -- (346,120) -- cycle ;
\draw  [dash pattern={on 0.84pt off 2.51pt}]  (446,70) -- (526,120) ;
\draw  [dash pattern={on 0.84pt off 2.51pt}]  (546,20) -- (486,50) ;
\draw  [dash pattern={on 0.84pt off 2.51pt}]  (546,20) -- (526,120) ;
\draw  [dash pattern={on 0.84pt off 2.51pt}]  (140,50) -- (260,100) ;

\draw (32,102.4) node [anchor=north west][inner sep=0.75pt]    {$x^{*}$};
\draw (337,122.4) node [anchor=north west][inner sep=0.75pt]    {$x^{*}$};
\draw (179,102.4) node [anchor=north west][inner sep=0.75pt]    {$\overline{x}_{1}$};
\draw (528,123.4) node [anchor=north west][inner sep=0.75pt]    {$\tilde{x}_{1}$};
\draw (208,102.4) node [anchor=north west][inner sep=0.75pt]    {$x_{1}( \tau )$};
\draw (135,29.4) node [anchor=north west][inner sep=0.75pt]    {$\overline{x}_{j}$};
\draw (427,49.4) node [anchor=north west][inner sep=0.75pt]    {$\overline{x}_{2}$};
\draw (551,3.4) node [anchor=north west][inner sep=0.75pt]    {$\tilde{x}_{2}$};
\draw (457,27.4) node [anchor=north west][inner sep=0.75pt]    {$x_{2}( \tau )$};
\draw (263,87.4) node [anchor=north west][inner sep=0.75pt]    {$\tilde{x}_{1}$};

\end{tikzpicture}}
		\caption{Left: Moving $x_1(\tau)$ increases the distance. Right: the minimum distance when moving $x_2(\tau)$ is realized by the right triangle.}
		\label{f-angle2}
	\end{figure}
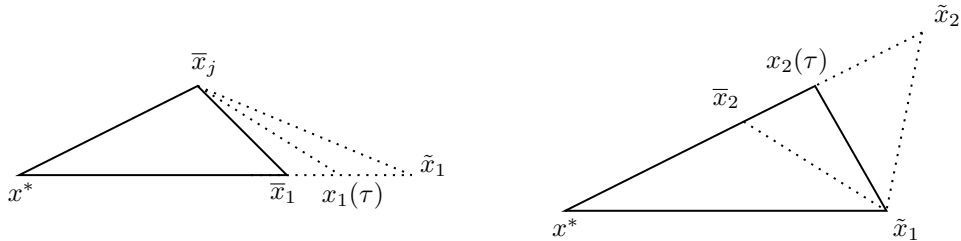
	
	We now choose one of the indices $i=2,\ldots N$ realizing the maximal distance $|x^*-\bar x_i|$, that we assume to be the index $2$, and move it along the quasi-static trajectory $x_2(\tau)=x^*+\bar x_2 +\tau(\tilde x_2-\bar x_2)$ 
{(and corresponding velocity $\bar v(\tau)=0$).}
It is then clear that the distance $|\bar x_j-x_2(\tau)|$ is increasing for all $j=3,\ldots N$, due to the same observation as above. Instead, the distance $|\tilde x_1-x_2(\tau)|$ can eventually decrease, up to the minimum that is realized when the triangle with vertices $x^*, x_2(\tau), \tilde x_1$ is right in $x_2(\tau)$,  see Figure \ref{f-angle2}, right. Such a minimal distance is thus larger than $R\sin(\theta)$, where $\theta$ is the minimal angle given above. By choosing $R>\FR{\mathcal{L}}/\sin(\theta)$, we are ensured that the minimal distance is greater than $\FR{\mathcal{L}}$. 
	
	Repeat the same construction for the indices $3$ to $N$ and hence steer all agents to a circle of radius $R$. By contruction, the pairwise distance is larger than $\FR{\mathcal{L}}$. Rearrange indices on the circle to have $\tilde x_i=R(\cos(\alpha_i),\sin(\alpha_i))$ with $0< \alpha_1<\alpha_2<\ldots \leq \alpha_N\leq2\pi$. Consider now the target equidistributed configuration with radius $R$: we simultaneously steer each $\tilde x_i$ to $\hat x_i=R(\cos(\tfrac{2i\pi}{N}),\sin(\tfrac{2i\pi}{N}))$ by using again a quasi-static deformation along the path
	$$
	x_i(\tau)=R(\cos((1-\tau)\alpha_i +\tau \tfrac{2i\pi}{N}),\sin((1-\tau)\alpha_i +\tau \tfrac{2i\pi}{N})) 
	$$
{(and corresponding velocity $\bar v(\tau)=0$).}
	It is easy to verify that the simultaneous displacement along the circle ensures that the minimal distance $|x_i(\tau)-x_j(\tau)|$ is realized either at the beginning or at the end of the deformation, hence in all cases $|x_i(\tau)-x_j(\tau)|>\FR{\mathcal{L}}$.

	\subsubsection{Proof of Step 3} 
	At the end of Step 3, we have steered the system to a circular equidistributed configuration $(\hat x,0)$ of large radius $R>\bar R$ in $\tilde \Omega_\eps$. Since pairwise distances are arbitrarily large, we also have $(\hat x,0)\in\Omega_{\eps'}$ for an arbitrarily small $\eps'>0$, i.e., we can go back to the original system \eqref{syst} with the potential $W$. Since the system \eqref{syst} is locally controllable around $(\hat x,0)$, we can control it to a desired configuration, as follows.
		
	\medskip

	\paragraph{\bf Reaching a flock ring}
	{We first steer the system to $(\hat x,\delta \bar v)$ with $\bar v$ the desired unitary velocity direction and $\delta>0$ sufficiently small: this can be achieved exactly by local controllability.} We then let the system evolve by choosing $u_i=F_i$, which ensures that all velocities satisfy $v_i(t)=v_j(t)$, hence $x_i(t)-x_j(t)$ keeps being constant, hence $F_i\simeq 0$ along the motion. This also ensures that all velocities $v_i$ converge to $\sqrt\frac\alpha\beta \bar v$, as in Step 3.1 in the proof of Theorem \ref{mainthm}.
	
	We next reduce the flock radius. Note that each flock ring of radius $r$ can be realized as a trajectory of the system \eqref{syst}, provided that the control can be chosen as $u_i=F_i$. In the case of the flock ring, we have
	$
	|x_i(t)-x_j(t)|\geq |x_i(t)-x_{i+1}(t)|=2\sin(\tfrac{\pi}N) \bar R
	$
	and thus
	$$
	\|F_i\|\leq \sup_{r>2\sin(\tfrac{\pi}N) \bar R} |U'(r)| .
	$$
	The condition $\|u_i\|>\sup_{r>2\sin(\tfrac{\pi}N) \bar R} |U'(r)|$ ensures that each flock ring of radius $r\geq R$ can indeed be realized as a trajectory of the system. Hence, by a quasi-static deformation, we can steer the system from a flock of radius $r_1$ to a flock of radius $r_2$ whenever $r_1,r_2\geq \bar R$.
		
	\medskip

	\paragraph{\bf Reaching a mill ring} 
	We first steer the system to $(\hat x_i,\hat v_i)$ with $\hat v_i=\eps x_i^\perp$ for some small $\varepsilon>0$, by local controllability around $(\hat x_i,0)\in\Omega_\eps$. We choose 
	$
	u_i=F_i-\frac{|v_i|^2}{R} \frac{x_i}{|x_i|}
	$
	to ensure that each agent undergoes the correct centripetal force and moves along the circle of radius $R$. Note that $\|u\|\leq \eps+\frac{\alpha}{\beta R}\leq 2\eps$ due to $R>\frac{\alpha}{\beta \eps}$, hence the control is arbitrarily small, as in Step 3.3 in the proof of Theorem \ref{mainthm}.
	
	Similarly to the flock ring, we next reduce the radius to a chosen $r\geq \bar R$. Note that each mill ring with such a radius is a trajectory of \eqref{syst}, provided that $u_i=F_i-\frac{|v_i|^2}{r} \frac{x_i}{|x_i|}$. By symmetry of the configuration, both $F_i$ and $-\frac{|v_i|^2}{r} \frac{x_i}{|x_i|}$ are radial forces. Hence
	$$
	\|u_i\|\leq \|F_i\|+\frac{\alpha}{\beta r}\leq \sup_{r>2\sin(\tfrac{\pi}N) \bar R} |U'(r)|+\frac{\alpha}{\beta \bar R}\leq M,
	$$ 
	which ensures that any mill of radius $r\geq \bar R$ can be reached. Using a quasi-static deformation, we can steer the system from any mill ring of radius $r_1$ to any mill ring of radius $r_2$, whenever $r_1,r_2\geq \bar R$.

	\section{Concluding remarks}
	In this work, we have studied the controlled transition to and between flocks and mills for second-order swarming systems. We have shown that, through a combination of stability properties of the dynamics, the choice of suitable Lyapunov functions and the use of quasi-static deformations, it is possible to construct feedback laws steering the dynamics towards flocking and milling regimes. We have provided an effective optimisation-based synthesis of instantaneous controls guiding the dynamics to different basins of attraction, where self-organization achieves the desired configuration.
	
	The present work opens different research perspectives. So far, we have restricted our construction to instantaneous optimal feedback laws, which are fast to compute, but lack anticipation properties which are fundamental to truly benefit from the self-organization features of the dynamics. In particular, the use of a finite horizon optimal control, along with a choice of a sparse or total-variation control penalty, might induce a more parsimonious control action, acting over a reduced number of agents and time instances, to steer the swarm towards a self-organization regime before switching off.
		
	Another challenge appears as the number of agents in the swarm increases, a natural scenario in agent-based models. Then, the computation of finite-horizon feedback laws even if theoretically possible, see Remark \ref{rem:mf}, becomes prohibitively expensive, and it is necessary to resort to open-loop controls embedded into a model predictive control scheme. An alternative to this problem is to consider a mean-field approximation of the dynamics, working with a density of agents $\rho=\rho(x,v,t)$ instead of the microscopic state of the swarm, and reformulating the control design at this level. However, the synthesis of optimal control laws for transitioning  between mean-field flocks and mills is an open problem and may be the subject of future research.

	\bibliographystyle{abbrv}
	\bibliography{CKRTref}

\begin{thebibliography}{10}

\bibitem{ABCV}
G.~Albi, D.~Balagu{\'e}, J.~A. Carrillo, and J.~von Brecht.
\newblock Stability analysis of flock and mill rings for second order models in
  swarming.
\newblock {\em SIAM J. Appl. Math.}, 74(3):794--818, 2014.

\bibitem{ABCK}
G.~Albi, M.~Bongini, E.~Cristiani, and D.~Kalise.
\newblock Invisible control of self-organizing agents leaving unknown
  environments.
\newblock {\em SIAM J. Appl. Math.}, 76(4):1683--1710, 2016.

\bibitem{ACFK17}
G.~Albi, Y.-P. Choi, M.~Fornasier, and D.~Kalise.
\newblock Mean field control hierarchy.
\newblock {\em Appl. Math. Optim.}, 76(1):93--135, 2017.

\bibitem{albi2021momentdriven}
G.~Albi, M.~Herty, D.~Kalise, and C.~Segala.
\newblock Moment-driven predictive control of mean-field collective dynamics.
\newblock {\em arXiv preprint arXiv:2101.01970}, 2021.

\bibitem{ALBI201886}
G.~Albi and D.~Kalise.
\newblock ({S}ub)optimal feedback control of mean field multi-population
  dynamics.
\newblock {\em IFAC-PapersOnLine}, 51(3):86--91, 2018.

\bibitem{APZ14}
G.~Albi, L.~Pareschi, and M.~Zanella.
\newblock Boltzmann-type control of opinion consensus through leaders.
\newblock {\em Philos. Trans. R. Soc. Lond. Ser. A Math. Phys. Eng. Sci.},
  372(2028):20140138, 18, 2014.

\bibitem{Aoki}
I.~Aoki.
\newblock A simulation study on the schooling mechanism in fish.
\newblock {\em Bull. Japan Soc. Sci. Fish}, 48:1081--1088, 1982.

\bibitem{BAILO20181}
R.~Bailo, M.~Bongini, J.~A. Carrillo, and D.~Kalise.
\newblock Optimal consensus control of the cucker-smale model.
\newblock {\em IFAC-PapersOnLine}, 51(13):1 -- 6, 2018.

\bibitem{BCLR2}
D.~Balagu{\'e}, J.~A. Carrillo, T.~Laurent, and G.~Raoul.
\newblock Dimensionality of local minimizers of the interaction energy.
\newblock {\em Arch. Ration. Mech. Anal.}, 209(3):1055--1088, 2013.

\bibitem{BCLR}
D.~Balagu{\'e}, J.~A. Carrillo, T.~Laurent, and G.~Raoul.
\newblock Nonlocal interactions by repulsive-attractive potentials: radial
  ins/stability.
\newblock {\em Phys. D}, 260:5--25, 2013.

\bibitem{alethea}
A.~B.~T. Barbaro, K.~Taylor, P.~F. Trethewey, L.~Youseff, and B.~Birnir.
\newblock Discrete and continuous models of the dynamics of pelagic fish:
  application to the capelin.
\newblock {\em Math. Comput. Simulation}, 79(12):3397--3414, 2009.

\bibitem{bertozzi}
A.~L. Bertozzi, T.~Kolokolnikov, H.~Sun, D.~Uminsky, and J.~Von~Brecht.
\newblock Ring patterns and their bifurcations in a nonlocal model of
  biological swarms.
\newblock {\em Comm. Math. Sci.}, 13(4), 2015.

\bibitem{Bjorn}
B.~Birnir.
\newblock An {ODE} model of the motion of pelagic fish.
\newblock {\em J. Stat. Phys.}, 128:535--568, 2007.

\bibitem{BFK15}
M.~Bongini, M.~Fornasier, and D.~Kalise.
\newblock ({U}n)conditional consensus emergence under perturbed and
  decentralized feedback controls.
\newblock {\em Discrete Contin. Dyn. Syst.}, 35(9):4071--4094, 2015.

\bibitem{Suttida1}
A.~Borzi and S.~Wongkaew.
\newblock Modeling and control through leadership of a refined flocking system.
\newblock {\em Math. Models Methods Appl. Sci.}, 25(02):255--282, 2015.

\bibitem{BressanPiccoli}
A.~Bressan and B.~Piccoli.
\newblock {\em Introduction to the Mathematical Theory of Control}, volume~2 of
  {\em AIMS Series on Applied Mathematics}.
\newblock American Institute of Mathematical Sciences (AIMS), Springfield, MO,
  2007.

\bibitem{Camazine}
S.~Camazine, J.~Deneubourg, N.~R. Franks, J.~Sneyd, G.~Theraulaz, and
  E.~Bonabeau.
\newblock {\em Self-Organization in Biological Systems}.
\newblock Princeton University Press, Princeton, 2001.

\bibitem{CFPT2013}
M.~Caponigro, M.~Fornasier, B.~Piccoli, and E.~Tr\'elat.
\newblock Sparse stabilization and optimal control of the {C}ucker-{S}male
  model.
\newblock {\em Math. Control Rel. Fields}, 3(4):447--466, 2013.

\bibitem{CFPT2015}
M.~Caponigro, M.~Fornasier, B.~Piccoli, and E.~Tr\'elat.
\newblock Sparse stabilization and control of alignment models.
\newblock {\em Math. Mod. Meth. Appl. Sci.}, 25(03):521--564, 2015.

\bibitem{CPRT2}
M.~Caponigro, B.~Piccoli, F.~Rossi, and E.~Tr{\'e}lat.
\newblock Mean-field sparse {J}urdjevic--{Q}uinn control.
\newblock {\em Math. Models Methods Appl. Sci.}, 27(07):1223--1253, 2017.

\bibitem{CPRT_A2017}
M.~Caponigro, B.~Piccoli, F.~Rossi, and E.~Tr\'{e}lat.
\newblock Sparse {J}urdjevic-{Q}uinn stabilization of dissipative systems.
\newblock {\em Automatica J. IFAC}, 86:110--120, 2017.

\bibitem{CDM}
J.~A. Carrillo, M.~G. Delgadino, and A.~Mellet.
\newblock Regularity of local minimizers of the interaction energy via obstacle
  problems.
\newblock {\em Comm. Math. Phys.}, 343(3):747--781, 2016.

\bibitem{Dorsogna3}
J.~A. Carrillo, M.~R. D'Orsogna, and V.~Panferov.
\newblock Double milling in self-propelled swarms from kinetic theory.
\newblock {\em Kin. Rel. Mod.}, 2:363--378, 2009.

\bibitem{CFRT}
J.~A. Carrillo, M.~Fornasier, J.~Rosado, and G.~Toscani.
\newblock Asymptotic flocking dynamics for the kinetic {C}ucker-{S}male model.
\newblock {\em SIAM J. Math. Anal.}, 42(1):218--236, 2010.

\bibitem{CFTV}
J.~A. Carrillo, M.~Fornasier, G.~Toscani, and F.~Vecil.
\newblock Particle, kinetic, and hydrodynamic models of swarming.
\newblock In {\em Mathematical modeling of collective behavior in
  socio-economic and life sciences}, Model. Simul. Sci. Eng. Technol., pages
  297--336. Birkh\"auser, 2010.

\bibitem{CHM2}
J.~A. Carrillo, Y.~Huang, and S.~Martin.
\newblock Explicit flock solutions for {Q}uasi-{M}orse potentials.
\newblock {\em European J. Appl. Math.}, 25(5):553--578, 2014.

\bibitem{CHM}
J.~A. Carrillo, Y.~Huang, and S.~Martin.
\newblock Nonlinear stability of flock solutions in second-order swarming
  models.
\newblock {\em Nonlinear Anal. Real World Appl.}, 17:332--343, 2014.

\bibitem{CKPP19}
Y.-P. Choi, D.~Kalise, J.~Peszek, and A.~A. Peters.
\newblock A collisionless singular {C}ucker--{S}male model with decentralized
  formation control.
\newblock {\em SIAM J. Appl. Dyn. Sys.}, 18(4):1954--1981, 2019.

\bibitem{CKP21}
Y.-P. Choi, D.~Kalise, and A.~A. Peters.
\newblock Collisionless and decentralized formation control for strings, 2021.
\newblock arXiv:2102.13621.

\bibitem{Dorsogna2}
Y.~Chuang, M.~R. D'Orsogna, D.~Marthaler, A.~Bertozzi, and L.~Chayes.
\newblock State transitions and the continuum limit for interacting,
  self-propelled particles.
\newblock {\em Phys. D}, 232:33--47, 2007.

\bibitem{coron-trelat04}
J.-M. Coron and E.~Tr{\'e}lat.
\newblock Global steady-state controllability of one-dimensional semilinear
  heat equations.
\newblock {\em SIAM J. Control Optim.}, 43(2):549--569 (electronic), 2004.

\bibitem{coron-trelat06}
J.-M. Coron and E.~Tr{\'e}lat.
\newblock Global steady-state stabilization and controllability of 1{D}
  semilinear wave equations.
\newblock {\em Commun. Contemp. Math.}, 8(4):535--567, 2006.

\bibitem{CuckerSmale1}
F.~Cucker and S.~Smale.
\newblock Emergent behavior in flocks.
\newblock {\em IEEE Trans. Automat. Control}, 52(5):852--862, 2007.

\bibitem{CuckerSmale2}
F.~Cucker and S.~Smale.
\newblock On the mathematics of emergence.
\newblock {\em Jpn. J. Math.}, 2(1):197--227, 2007.

\bibitem{Dorsogna}
M.~R. D'Orsogna, Y.-L. Chuang, A.~L. Bertozzi, and L.~S. Chayes.
\newblock Self-propelled particles with soft-core interactions: patterns,
  stability, and collapse.
\newblock {\em Physical review letters}, 96(10):104302, 2006.

\bibitem{HaLiu}
S.-Y. Ha and J.-G. Liu.
\newblock A simple proof of the {C}ucker-{S}male flocking dynamics and
  mean-field limit.
\newblock {\em Commun. Math. Sci.}, 7(2):297--325, 2009.

\bibitem{HaTa}
S.-Y. Ha and E.~Tadmor.
\newblock From particle to kinetic and hydrodynamic descriptions of flocking.
\newblock {\em Kinet. Relat. Models}, 1(3):415--435, 2008.

\bibitem{Hemelrijk:Hildenbrandt}
C.~K. Hemelrijk and H.~Hildenbrandt.
\newblock Self-organized shape and frontal density of fish schools.
\newblock {\em Ethology}, 114:245--254, 2008.

\bibitem{HK18}
M.~{Herty} and D.~{Kalise}.
\newblock Suboptimal nonlinear feedback control laws for collective dynamics.
\newblock In {\em Proceedings 2018 IEEE 14th ICCA}, pages 556--561, 2018.

\bibitem{Huth:Wissel}
A.~Huth and C.~Wissel.
\newblock The simulation of fish schools in comparison with experimental data.
\newblock {\em Ecol. Model.}, 75/76:135--145, 1994.

\bibitem{JQ}
V.~{J}urdjevic and J.~P. {Q}uinn.
\newblock Controllability and stability.
\newblock {\em J. Differential Equations}, 28(3):381--389, 1978.

\bibitem{khalil}
H.~K. Khalil.
\newblock {\em {Nonlinear systems; 3rd ed.}}
\newblock Prentice-Hall, 2002.

\bibitem{KZ20}
D.~Ko and E.~Zuazua.
\newblock Asymptotic behavior and control of a ``guidance by repulsion'' model.
\newblock {\em Math. Models Methods Appl. Sci.}, 30(04):765--804, 2020.

\bibitem{KochWhite}
A.~L. Koch and D.~White.
\newblock The social lifestyle of myxobacteria.
\newblock {\em BioEssays}, 20(12):1030--1038, 1998.

\bibitem{survey}
T.~Kolokolnikov, J.~A. Carrillo, A.~Bertozzi, R.~Fetecau, and M.~Lewis.
\newblock Emergent behaviour in multi-particle systems with non-local
  interactions [{E}ditorial].
\newblock {\em Phys. D}, 260:1--4, 2013.

\bibitem{KSUB}
T.~Kolokonikov, H.~Sun, D.~Uminsky, and A.~Bertozzi.
\newblock Stability of ring patterns arising from 2d particle interactions.
\newblock {\em Physical Review E}, 84(1):015203, 2011.

\bibitem{LeeMarkus}
E.~B. Lee and L.~Markus.
\newblock {\em Foundations of optimal control theory}.
\newblock John Wiley \& Sons, Inc., New York-London-Sydney, 1967.

\bibitem{Levine:2000}
H.~Levine, W.-J. Rappel, and I.~Cohen.
\newblock Self-organization in systems of self-propelled particles.
\newblock {\em Phys. Rev. E}, 63:017101, Dec 2000.

\bibitem{lukeman}
R.~Lukeman, Y.~Li, and L.~Edelstein-Keshet.
\newblock Inferring individual rules from collective behavior.
\newblock {\em Proc. Natl. Acad. Sci. U.S.A.}, 107(28):12576--12580, 2010.

\bibitem{mt}
S.~Motsch and E.~Tadmor.
\newblock Heterophilious dynamics enhances consensus.
\newblock {\em SIAM Rev.}, 56(4):577--621, 2014.

\bibitem{science1999}
J.~Parrish and L.~Edelstein-Keshet.
\newblock Complexity, pattern, and evolutionary trade-offs in animal
  aggregation.
\newblock {\em Science}, 284(5411):99 --101, 1999.

\bibitem{PPT2019}
B.~Piccoli, N.~Pouradier~Duteil, and E.~Tr\'{e}lat.
\newblock Sparse control of {H}egselmann-{K}rause models: black hole and
  declustering.
\newblock {\em SIAM J. Control Optim.}, 57(4):2628--2659, 2019.

\bibitem{PR18}
B.~Piccoli and F.~Rossi.
\newblock Measure-theoretic models for crowd dynamics.
\newblock In {\em Crowd Dynamics, Volume 1}, pages 137--165. Springer, 2018.

\bibitem{PRT2015}
B.~Piccoli, F.~Rossi, and E.~Tr\'{e}lat.
\newblock Control to flocking of the kinetic {C}ucker-{S}male model.
\newblock {\em SIAM J. Math. Anal.}, 47(6):4685--4719, 2015.

\bibitem{PoucholTrelatZuazua}
C.~Pouchol, E.~Tr\'{e}lat, and E.~Zuazua.
\newblock Phase portrait control for 1{D} monostable and bistable
  reaction-diffusion equations.
\newblock {\em Nonlinearity}, 32(3):884--909, 2019.

\bibitem{Sontag}
E.~D. Sontag.
\newblock {\em Mathematical control theory}, volume~6.
\newblock Springer, 2013.

\bibitem{spong}
M.~W. Spong, S.~Hutchinson, and M.~Vidyasagar.
\newblock {\em Robot modeling and control}.
\newblock John Wiley \& Sons, 2nd edition, 2004.

\bibitem{trelatbook2005}
E.~Tr\'{e}lat.
\newblock {\em Contr\^{o}le optimal}.
\newblock Math\'{e}matiques Concr\`etes. [Concrete Mathematics]. Vuibert,
  Paris, 2005.
\newblock Th\'{e}orie \& applications. [Theory and applications].

\bibitem{BUKB}
J.~von Brecht, D.~Uminsky, T.~Kolokolnikov, and A.~Bertozzi.
\newblock Predicting pattern formation in particle interactions.
\newblock {\em Math. Mod. Meth. Appl. Sci.}, 22:1140002, 2012.

\bibitem{Suttida2}
S.~Wongkaew, M.~Caponigro, and A.~Borzi.
\newblock On the control through leadership of the {H}egselmann--{K}rause
  opinion formation model.
\newblock {\em Math. Models Methods Appl. Sci.}, 25(03):565--585, 2015.

\end{thebibliography}
	
\end{document}